\documentclass{amsart}
\usepackage{amsthm, amssymb}
\usepackage{tabmac}
\usepackage{pstricks,pst-node}

\psset{unit=1pt} \psset{arrowsize=4pt 1} \psset{linewidth=1pt}
\newrgbcolor{mygreen}{.2 .7 .2}
\newrgbcolor{myred}{.7 .3 .2}
\newgray{mygray}{.90}
\countdef\x=23 \countdef\y=24 \countdef\z=25 \countdef\t=26

\newtheorem{theorem}{Theorem}
\newtheorem{corollary}[theorem]{Corollary}
\newtheorem{prop}[theorem]{Proposition}
\newtheorem{proposition}[theorem]{Proposition}
\newtheorem{lemma}[theorem]{Lemma}
\newtheorem{lem}[theorem]{Lemma}
\newtheorem{conjecture}[theorem]{Conjecture}
\newtheorem{example}{Example}
\newtheorem{definition}{Definition}
\newtheorem{remark}[theorem]{Remark}

\newtheorem{exercise}{Exercise}

\numberwithin{theorem}{section}
\numberwithin{equation}{section}

\title{Stanley symmetric functions \\ and  \\Peterson algebras}
\author{Thomas Lam}
\email{tfylam@umich.edu}
 \thanks{The author was supported by NSF grants DMS-0652641 and DMS-0901111, and by a Sloan Fellowship.}
\date{July 16, 2010}

\def\A{{\mathbb A}}
\def\B{{\mathbb B}}

\def\a{{\mathbf{a}}}
\def\af{{\mathrm{af}}}
\def\afA{{\mathbb A}_{\mathrm{af}}}
\def\ah{\tilde \h}
\def\aNC{(\A_\af)_0}
\def\Pet{{\mathbb P}}
\def\aB{{\mathbb B}_{\af}}
\def\aI{I_\af}
\def\al{\alpha}
\def\aPsi{\Xi_\af}

\def\aW{W_\af}
\def\Bo{{\mathcal B}}
\def\Des{\mathrm{Des}}
\def\fW{W}
\def\fI{I}

\def\Frac{{\mathrm{Frac}}}
\def\Fun{{\mathrm{Fun}}}
\def\Gr{{\mathrm{Gr}}}
\def\h{{\mathbf{h}}}
\def\i{{\mathbf{i}}}
\def\id{\mathrm{id}}
\def\j{{\mathbf{j}}}
\def\la{\lambda}
\def\La{\Lambda}

\def\Q{{\mathbb Q}}
\def\om{\omega}
\def\s{{\mathbf{s}}}
\def\Sym{{\mathrm{Sym}}}

\def\tF{\tilde F}
\def\tS{\tilde S}
\def\wt{{\mathrm{wt}}}
\def\Z{{\mathbb Z}}

\newcommand\os[1]{\overline{#1}\;}
\newcommand\ip[1]{\langle #1 \rangle}

\begin{document}
\begin{abstract} These are (mostly) expository notes for lectures on affine Stanley symmetric functions given at the Fields Institute in 2010.  We focus on the algebraic and combinatorial parts of the theory.  The notes contain a number of exercises and open problems.
\end{abstract}

\maketitle
Stanley symmetric functions are a family $\{F_w \mid w \in S_n\}$ of symmetric functions indexed by permutations.  They were invented by Stanley \cite{Sta} to enumerate the reduced words of elements of the symmetric group.  The most important properties of the Stanley symmetric functions are their symmetry, established by Stanley, and their Schur positivity, first proven by Edelman and Greene \cite{EG}, and by Lascoux and Sch\"{u}tzenberger \cite{LS:Schub}.

Recently, a generalization of Stanley symmetric functions to affine permutations was developed in \cite{Lam:affStan}.  These affine Stanley symmetric functions turned out to have a natural geometric interpretation \cite{Lam:Schub}: they are pullbacks of the cohomology Schubert classes of the affine flag variety $LSU(n)/T$ to the affine Grassmannian (or based loop space) $\Omega SU(n)$ under the natural map $\Omega SU(n) \to LSU(n)/T$.  The combinatorics of reduced words and the geometry of the affine homogeneous spaces are connected via the nilHecke ring of Kostant and Kumar \cite{KK}, together with a remarkable commutative subalgebra due to Peterson \cite{Pet}.    The symmetry of affine Stanley symmetric functions follows from the commutativity of Peterson's subalgebra, and the positivity in terms of affine Schur functions is established via the relationship between affine Schubert calculus and quantum Schubert calculus \cite{LS:QH,LL}.  The affine-quantum connection was also discovered by Peterson.  

The affine generalization also connects Stanley symmetric functions with the theory of Macdonald polynomials \cite{Mac} -- my own involvement in this subject began when I heard a conjecture of Mark Shimozono relating the Lapointe-Lascoux-Morse $k$-Schur functions \cite{LLM} to the affine Grassmannian.

While the definition of (affine) Stanley symmetric functions does not easily generalize to other (affine) Weyl groups (see \cite{BH,BL,FK:C,LSS:C,Pon}), the algebraic and geometric constructions mentioned above do.

This article introduces Stanley symmetric functions and affine Stanley symmetric functions from the combinatorial and algebraic point of view.  The goal is to develop the theory (with the exception of positivity) without appealing to geometric reasoning.  The notes are aimed at an audience with some familiarity with symmetric functions, Young tableaux and Coxeter groups/root systems.

The first third (Sections \ref{sec:combin} - \ref{sec:affine}) of the article centers on the combinatorics of reduced words.  We discuss reduced words in the (affine) symmetric group, the definition of the (affine) Stanley symmetric functions, and introduce the Edelman-Greene correspondence.  Section \ref{sec:Weyl} reviews the basic notation of Weyl groups and affine Weyl groups.  In Sections \ref{sec:algebra}-\ref{sec:finiteFS} we introduce and study four algebras: the nilCoxeter algebra, the Kostant-Kumar nilHecke ring, the Peterson centralizer subalgebra of the nilHecke ring, and the Fomin-Stanley subalgebra of the nilCoxeter algebra.  The discussion in Section \ref{sec:finiteFS} is new, and is largely motivated by a conjecture (Conjecture \ref{conj:LP}) of the author and Postnikov.  In Section \ref{sec:geom}, we give a list of geometric interpretations and references for the objects studied in the earlier sections.

We have not intended to be comprehensive, especially with regards to generalizations and variations.  There are four such which we must mention: 
\begin{enumerate} 
\item
There is an important and well-developed connection between Stanley symmetric functions and Schubert polynomials, see \cite{BJS, LS:Schub}.  
\item
There is a theory of (affine) Stanley symmetric functions in classical types; see  \cite{BH,BL,FK:C,LSS:C,Pon}.
\item
Nearly all the constructions here have $K$-theoretic analogues.  For full details see \cite{Buc, BKSTY,FK:K,LSS}.
\item
There is a $t$-graded version of the theory.  See \cite{LLM,LM:core,LM:ktab}.
\end{enumerate}

We have included exercises and problems throughout which occasionally assume more prerequisites.  The exercises vary vastly in terms of difficulty.  Some exercises essentially follow from the definitions, but other problems are questions for which I do not know the answer to.  

\section{Stanley symmetric functions and reduced words}
\label{sec:combin}
For an integer $m \geq 1$, let $[m] = \{1,2,\ldots,m\}$.  For a partition (or composition) $\la = (\la_1,\la_2,\ldots,\la_\ell)$, we write $|\la| = \la_1 + \cdots + \la_\ell$.  The {\it dominance order} $\preceq$ on partitions is given by $\la \prec \mu$ if for some $k > 0 $ we have $\la_1 + \la_2 + \cdots + \la_j = \mu_1 + \mu_2 + \cdots + \mu_j$ for $1 \leq j <k$ and $\la_1 + \la_2 + \cdots + \la_k < \mu_1 + \mu_2 + \cdots + \mu_k$.  The {\it descent set} $\Des(\a)$ of a word $a_1a_2\cdots a_n$ is given by $\Des(\a) = \{i \in [n-1] \mid a_i > a_{i+1}\}$.

\subsection{Young tableaux and Schur functions}
We shall assume the reader has some familiarity with symmetric functions and Young tableaux \cite{Mac} \cite[Ch. 7]{EC2}.  We write $\La$ for the ring of symmetric functions.  We let $m_\la$, where $\la$ is a partition, denote the monomial symmetric function, and let $h_k$ and $e_k$,  for integers $k \geq 1$, denote the homogeneous and elementary symmetric functions respectively.  For a partition $\la = (\la_1,\la_2,\ldots,\la_\ell)$, we define $h_\la:= h_{\la_1}h_{\la_2}\cdots h_{\la_\ell}$, and similarly for $e_\la$.  We let $\ip{.,.}$ denote the Hall inner product of symmetric functions.  Thus $\ip{h_\la,m_\mu} = \ip{m_\la,h_\mu}=\ip{s_\la,s_\mu} = \delta_{\la\mu}$.

We shall draw Young diagrams in English notation.  A tableau of shape $\la$ is a filling of the Young diagram of $\la$ with integers.  A tableau is {\it column-strict} (resp. {\it row-strict}) if it is increasing along columns (resp. rows).  A tableau is {\it standard} if it is column-strict and row-strict, and uses each number $1,2,\ldots,|\la|$ exactly once.  A tableau is {\it semi-standard} if it is column-strict, and weakly increasing along rows.  Thus the tableaux
$$
\tableau[sY]{1&2&4&5\\3&6\\7} \qquad \tableau[sY]{1&1&2&3\\4&4\\6}
$$
are standard and semistandard respectively.  The weight $\wt(T)$ of a tableau $T$ is the composition $(\alpha_1,\alpha_2,\ldots)$ where $\alpha_i$ is equal to the number of $i$'s in $T$.  The Schur function $s_\la$ is given by
$$
s_\la(x_1,x_2,\ldots) = \sum_T x^{\wt(T)}
$$
where the summation is over semistandard tableaux of shape $\la$, and for a composition $\alpha$, we define $x^\alpha:=x_1^{\alpha_1}x_2^{\alpha_2} \cdots$.  For a standard Young tableau $T$ we define $\Des(T) = \{i \mid \text{$i+1$ is in a lower row than $i$}\}$.  We also write $f^\la$ for the number of standard Young tableaux of shape $\la$.  Similar definitions hold for skew shapes $\la/\mu$.

We shall often use the Jacobi-Trudi formula for Schur functions (see \cite{Mac,EC2}).
\begin{theorem}\label{thm:JT} 
Let $\la = (\la_1\geq \la_2 \geq \cdots \geq \la_\ell >0)$ be a partition.  Then
$$
s_\la = {\rm det}(h_{\la_i+j-i})_{i,j=1}^{\ell}.
$$
\end{theorem}

\subsection{Permutations and reduced words}\label{sec:perm}
Let $S_n$ denote the symmetric group of permutations on the letters $[n]$.  We think of permutations $w,v \in S_n$ as bijections $[n] \to [n]$, so that the product $w \, v \in S_n$ is the composition $w \circ v$ as functions.  The simple transposition $s_i \in S_n$, $i \in \{1,2,\ldots,n-1\}$ swaps the letters $i$ and $i+1$, keeping the other letters fixed.  The symmetric group is generated by the $s_i$, with the relations
\begin{align*}
s_i^2 &= 1 & \mbox{for $1 \leq i \leq n-1$}\\
s_i s_{i+1} s_i &=s_{i+1} s_i s_{i+1} &\mbox{for $1 \leq i \leq n-2$}\\
s_i s_j &= s_j s_i &\mbox{for $|i-j|>1$}
\end{align*}
The {\it length} $\ell(w)$ of a permutation $w \in S_n$ is the length of the shortest expression $w = s_{i_1} \cdots s_{i_\ell}$ for $w$ as a product of simple generators.  Such a shortest expression is called a {\it reduced expression} for $w$, and the word $i_1 i_2 \cdots i_\ell$ is a {\it reduced word} for $w$.  Let $R(w)$ denote the set of reduced words of $w \in S_n$.  We usually write permutations in one-line notation, or alternatively give reduced words.  For example $3421 \in S_4$ has reduced word $23123$.

There is a natural embedding $S_n \hookrightarrow S_{n+1}$ and we will sometimes not distinguish between $w \in S_n$ and its image in $S_{n+1}$ under this embedding.  

\subsection{Reduced words for the longest permutation}
The longest permutation $w_0 \in S_n$ is $w_0 = n\, (n-1)\, \cdots \,2\,1$ in one-line notation.  Stanley \cite{Sta} conjectured the following formula for the number of reduced words of $w_0$, which he then later proved using the theory of Stanley symmetric functions.  Let $\delta_n = (n,n-1,\ldots,1)$ denote the staircase of size $n$.
\begin{theorem}[\cite{Sta}]\label{thm:longestword}
The number $R(w_0)$ of reduced words for $w_0$ is equal to the number $f^{\delta_{n-1}}$ of staircase shaped standard Young tableaux.
\end{theorem}

\subsection{The Stanley symmetric function}
\begin{definition}[Original definition] \label{def:storig}
Let $w \in S_n$.  Define the Stanley symmetric function\footnote{Our conventions differ from Stanley's original definitions by $w \leftrightarrow w^{-1}$.} $F_w$ by
$$
F_w(x_1,x_2,\ldots) = \sum_{\stackrel{a_1 a_2 \ldots a_\ell \in  R(w) \ \  1 \leq b_1 \leq b_2 \leq \cdots \leq b_\ell}{a_i < a_{i+1} \implies b_i < b_{i+1}}} x_{b_1} x_{b_2} \cdots x_{b_\ell}.
$$
\end{definition}

We shall establish the following fundamental result \cite{Sta} in two different ways in Sections \ref{sec:EG} and \ref{sec:algebra}, but shall assume it for the remainder of this section.
\begin{theorem}[\cite{Sta}]\label{thm:stsymm}
The generating function $F_w$ is a symmetric function.
\end{theorem}

A word $a_1 a_2 \cdots a_\ell$ is {\it decreasing} if $a_1 > a_2 > \cdots > a_\ell$.  A permutation $w \in S_n$ is decreasing if it has a (necessarily unique) decreasing reduced word.  The identity $\id \in S_n$ is considered decreasing.  A {\it decreasing factorization} of $w \in S_n$ is an expression $w = v_1 v_2 \cdots v_r$ such that $v_i \in S_n$ are decreasing, and $\ell(w) = \sum_{i=1}^r \ell(v_i)$.

\begin{definition}[Decreasing factorizations] \label{def:stdec}
Let $w \in S_n$.  Then
$$
F_w(x_1,x_2,\ldots) = \sum_{w = v_1 v_2 \cdots v_r} x_1^{\ell(v_1)} \cdots x_r^{\ell(v_r)}.
$$
\end{definition}

\begin{example}\label{ex:1323}
Consider $w = s_1s_3s_2s_3 \in S_4$.  Then $R(w) = \{1323,3123,1232\}$.  Thus
$$
F_w = m_{211}+3m_{1111} = s_{211}.
$$
The decreasing factorizations which give $m_{211}$ are $\overline{31}\;\overline{2}\;\overline{3}, \overline{1}\;\overline{32}\;\overline{3}, \overline{1}\;\overline{2}\;\overline{32}$.
\end{example}

\subsection{The code of a permutation}
Let $w \in S_n$.  The {\it code} $c(w)$ is the sequence $c(w) = (c_1, c_2, \ldots, )$ of nonnegative integers given by $c_i = \#\{j\in [n]\mid j > i \text{ and } w(j) < w(i)\}$ for $i \in [n]$, and $c_i = 0$ for $i > n$.  Note that the code of $w$ is the same regardless of which symmetric group it is considered an element of.

Let $\la(w)$ be the partition conjugate to the partition obtained from rearranging the parts of $c(w^{-1})$ in decreasing order.

\begin{example}
Let $w = 216534 \in S_6$.  Then $c(w) = (1,0,3,2,0,0,\ldots)$, and $c(w^{-1}) = (1,0,2,2,1,0,\ldots)$.  Thus $\la(w) = (4,2)$.
\end{example}

For a symmetric function $f \in \La$, let $[m_\la]f$ denote the coefficient of $m_\la$ in $f$.

\begin{prop}[\cite{Sta}] \label{prop:dominant} Let $w \in S_n$.
\begin{enumerate}
\item
Suppose $[m_\la]F_w \neq 0$.  Then $\la \prec \la(w)$.
\item
$[m_{\la(w)}]F_w = 1$.
\end{enumerate}
\end{prop}
\begin{proof}
Left multiplication of $w$ by $s_i$ acts on $c(w^{-1})$ by
$$
(c_1,\ldots , c_i,c_{i+1},\ldots)\longmapsto
(c_1,\ldots,c_{i+1},c_i-1,\ldots)$$
whenever $\ell(s_iw) < \ell(w)$.  Thus factorizing a decreasing permutation $v$ out of $w$ from the left will decrease $\ell(v)$ different entries of $c(w^{-1})$ each by $1$.  (1) follows easily from this observation.

To obtain (2), one notes that there is a unique decreasing permutation $v$ of length $\mu_1(w)$ such that $\ell(w) = \ell(v^{-1}w) + \ell(v)$.
\end{proof}

\begin{example}
Continuing Example \ref{ex:1323}, one has $w = 2431$ in one-line notation.  Thus $\la(w) =  (2,1,1)$, agreeing with Proposition \ref{prop:dominant} and our previous calculation.
\end{example}

\subsection{Fundamental Quasi-symmetric functions}
Let $D \subset [n-1]$.  Define the (Gessel) {\it fundamental quasi-symmetric function} $L_D$ by 
$$
L_D(x_1,x_2,\ldots) = \sum_{\stackrel{1 \leq b_1 \leq b_2 \cdots \leq b_n}{i \in D \implies b_{i+1} > b_i}} x_{b_1}x_{b_2} \cdots x_{b_n}.
$$
Note that $L_D$ depends not just on the set $D$ but also on $n$.

A basic fact relating Schur functions and fundamental quasi-symmetric functions is:
\begin{prop}\label{prop:Schurquasi}
Let $\la$ be  a partition.  Then
$$
s_\lambda = \sum_T L_{\Des(T)}.
$$
\end{prop}

\begin{definition}[Using quasi symmetric functions] \label{def:stquasi}
Let $w \in S_n$.  Then
$$
F_w(x_1,x_2,\ldots) = \sum_{\a \in R(w^{-1})} L_{\Des(\a)}.
$$
\end{definition}

\begin{example}
Continuing Example \ref{ex:1323}, we have $F_w = L_{2} + L_{1}+L_{3}$, where all subsets are considered subsets of $[3]$.  Note that these are exactly the descent sets of the
tableaux
$$
\tableau[sY]{1&2&3\\4} \qquad \tableau[sY]{1&2&4\\3} \qquad \tableau[sY]{1&3&4\\2}
$$
\end{example}


\subsection{Exercises}
\label{ssec:combinprob}
\begin{enumerate}
\item 
Prove that $|c(w)|:=\sum_i c_i(w)$ is equal to $\ell(w)$.
\item
Let $S_\infty = \cup_{n \geq 1} S_n$, where permutations are identified under the embeddings $S_1 \hookrightarrow S_2 \hookrightarrow S_3 \cdots$.  
Prove that $w \longmapsto c(w)$ is a bijection between $S_\infty$ and nonnegative integer sequences with finitely many non-zero entries.
\item
Prove the equivalence of Definitions \ref{def:storig}, \ref{def:stdec}, and \ref{def:stquasi}.  
\item
What happens if we replace decreasing factorizations by increasing factorizations in Definition \ref{def:stdec}?
\item
What is the relationship between $F_w$ and $F_{w^{-1}}$?
\item 
(Grassmannian permutations)
A permutation $w \in S_n$ is {\it Grassmannian} if it has at most one descent.
\begin{enumerate}
\item  Characterize the codes of Grassmannian permutations.
\item
Show that if $w$ is Grassmannian then $F_w$ is a Schur function. 
\item
Which Schur functions are equal to $F_w$ for some Grassmannian permutation $w \in S_n$?
\end{enumerate}
\item 
(321-avoiding permutations \cite{BJS})
A permutation $w \in S_n$ is {\it 321-avoiding} if there does not exist $a < b <c$ such that $w(a) > w(b) > w(c)$.  Show that $w$ is 321-avoiding if and only if no reduced word $\i \in R(w)$ contains a consecutive subsequence of the form $j (j+1) j$.  If $w$ is 321-avoiding, show directly from the definition that $F_w$ is a skew Schur function.
\end{enumerate}

\section{Edelman-Greene insertion}\label{sec:EG}
\subsection{Insertion for reduced words}
We now describe an insertion algorithm for reduced words, due to Edelman and Greene \cite{EG}, which establishes Theorem \ref{thm:stsymm}, and in addition stronger positivity properties.  Related bijections were studied by Lascoux-Sch\"utzenberger \cite{LS} and by Haiman \cite{Hai}.

Let $T$ be a column and row strict Young tableau.  The {\it reading word} $r(T)$ is the word obtained by reading the rows of $T$ from left to right, starting with the bottom row.

Let $w \in S_n$.  We say that a tableau $T$ is a {\it EG-tableau} for $w$ if $r(T)$ is a reduced word for $w$.  For example,
$$
T = \tableau[sY]{1&2&3\\2&3}
$$
has reading word $r(T) = 23123$, and is an EG-tableau for $3421 \in S_4$.

\begin{theorem}[\cite{EG}]\label{thm:EG}
Let $w \in S_n$.  There is a bijection between $R(w)$ and the set of pairs $(P,Q)$, where $P$ is an EG-tableau for $w$, and $Q$ is a standard Young tableau with the same shape as $P$.  Furthermore, under the bijection $\i \leftrightarrow (P(\i),Q(\i))$ we have $\Des(\i) = \Des(Q)$.
\end{theorem}

Combining Theorem \ref{thm:EG} with Proposition \ref{prop:Schurquasi} and Definition \ref{def:stquasi}, we obtain:
\begin{corollary}\label{cor:stpos}
Let $w \in S_n$.  Then $F_w = \sum_\lambda \al_{w \la} s_\la$, where 
$\al_{w\la}$ is equal to the number of EG-tableau for $w^{-1}$.  In particular, $F_w$ is Schur positive.
\end{corollary}
As a consequence we obtain Theorem \ref{thm:stsymm}.

\begin{lemma}\label{lem:EGshape}
Suppose $T$ is an EG-tableau for $S_n$.  Then the shape of $T$ is contained in the staircase $\delta_{n-1}$.
\end{lemma}
\begin{proof}
Since $T$ is row-strict and column-strict, the entry in the $i$-th row and $j$-th column is greater than or equal to $i+j-1$.  But EG-tableaux can only be filled with the numbers $1,2,\ldots,n-1$, so the shape of $T$ is contained inside $\delta_{n-1}$.
\end{proof}

\begin{proof}[Proof of Theorem \ref{thm:longestword}]
The longest word $w_0$ has length $\binom{n}{2}$.  Suppose $T$ is an EG-tableau for $w_0$.  Since the staircase $\delta_{n-1}$ has exactly $\binom{n}{2}$ boxes, Lemma \ref{lem:EGshape} shows that $T$ must have shape $\delta_{n-1}$.  But then it is easy to see that the only possibility for $T$ is the tableau
$$
\tableau[mbY]{1 & 2 & 3& \cdots & n-1\\2 & 3 & \cdots & n-1\\ 3&\cdots&n-1 \\ \vdots &\vdots \\ n-1}
$$
Thus it follows from Theorem \ref{thm:EG} that $R(w_0) = f^{\delta_{n-1}}$.
\end{proof}

The proof of Theorem \ref{thm:EG} is via an explicit insertion algorithm.  Suppose $T$ is an EG-tableau.  We describe the insertion of a letter $a$ into $T$.  If the largest letter in the first row of $T$ is less than $a$, then we add $a$ to the end of the first row, and the insertion is complete.  Otherwise, we find the smallest letter $a'$ in $T$ greater than $a$, and bump $a'$ to the second row, where the insertion algorithm is recursively performed.  The first row $R$ of $T$ changes as follows: if both $a$ and $a+1$ were present in $R$ (and thus $a' = a+1$) then the row remains unchanged; otherwise, we replace $a'$ by $a$ in $R$.

For a reduced word $\i = i_1\,i_2\, \cdots \, i_\ell$, we obtain $P(\i)$ by inserting $i_1$, then $i_2$, and so on, into the empty tableau.  The tableau $Q(\i)$ is the standard Young tableau which records the changes in shape of the EG-tableau as this insertion is performed.  

\begin{example}
Let $\i = 21232$.  Then the successive EG-tableau are
$$
\tableau[sY]{2} \qquad \tableau[sY]{1\\2} \qquad \tableau[sY]{1&2\\2} \qquad \tableau[sY]{1&2&3 \\ 2} \qquad \tableau[sY]{1&2&3 \\ 2&3}
$$
so that 
$$
Q(\i) = \tableau[sY]{1&3&4\\2&5}.
$$
\end{example}

\subsection{Coxeter-Knuth relations}
\label{sec:CK}
Let $\i$ be a reduced word.  A {\it Coxeter-Knuth relation} on $\i$ is one of the following transformations on three consecutive letters of $\i$:
\begin{enumerate}
\item
$a\, (a+1) \,a \sim (a+1)\, a \, (a+1)$
\item
$a \, b \, c \sim a \, c \, b$ when $b < a < c$
\item
$a \, b \, c \sim b \, a \, c$ when $b < c < a$
\end{enumerate}
Since Coxeter-Knuth relations are in particular Coxeter relations for the symmetric group, it follows that if two words are related by Coxeter-Knuth relations then they represent the same permutation in $S_n$.  The following result of Edelman and Greene states that Coxeter-Knuth equivalence is an analogue of Knuth-equivalence for reduced words.

\begin{theorem}[\cite{EG}]\label{thm:CK}
Suppose $\i,\i' \in R(w)$.  Then $P(\i)=P(\i')$ if and only if $\i$ and $\i'$ are Coxeter-Knuth equivalent.
\end{theorem}

\subsection{Exercises and Problems}\label{ssec:EGprob}
\begin{enumerate}
\item
For $w \in S_n$ let $1 \times w \in S_{n+1}$ denote the permutation obtained from $w$ by adding $1$ to every letter in the one-line notation, and putting a $1$ in front.  Thus if $w = 24135$, we have $1 \times w = 135246$.  Show that $F_w = F_{1 \times w}$.
\item
Suppose $w \in S_n$ is 321-avoiding (see Section \ref{ssec:combinprob}).  Show that Edelman-Greene insertion of $\i \in R(w)$ is the usual Robinson-Schensted insertion of $\i$.
\item (Vexillary permutations \cite{BJS})
A permutation $w \in S_n$ is {\it vexillary} if it avoids the pattern 2143.  That is, there do not exist $a < b < c < d$ such that $w(b) < w(a) < w(d) < w(c)$.  In particular, $w_0$ is vexillary.  

The Stanley symmetric function $\tF_w$ is equal to a Schur function $s_\la$ if and only if $w$ is vexillary \cite[p.367]{BJS}.  Is there a direct proof using Edelman-Greene insertion? 

\item (Shape of a reduced word)
The shape $\la(\i)$ of a reduced word $\i \in R(w)$ is the shape of the tableau $P(\i)$ or $Q(\i)$ under Edelman-Greene insertion.  Is there a direct way to read off the shape of a reduced word?  (See \cite{TY1} for a description of $\la_1(\i)$.)

For example, Greene's invariants (see for example \cite[Ch. 7]{EC2}) describe the shape of a word under Robinson-Schensted insertion.

\item (Coxeter-Knuth relations and dual equivalence (graphs))
\label{rem:CK}
Show that Coxeter-Knuth relations on reduced words correspond exactly to elementary dual equivalences on the recording tableau (see \cite{Hai}).  They thus give a structure of a dual equivalence graph \cite{Ass} on $R(w)$. 

An independent proof  of this (in particular not using EG-insertion), together with the technology of \cite{Ass}, would give a new proof of the Schur positivity of Stanley symmetric functions.

\item (Lascoux-Sch\"{u}tzenberger transition)
Let $(i,j) \in S_n$ denote the transposition which swaps $i$ and $j$.  Fix $r \in [n]$ and $w \in S_n$.  The Stanley symmetric functions satisfy \cite{LS} the equality
\begin{equation}\label{eq:transition}
\sum_{u=w\,(r,s): \,\,
\ell(u)=\ell(w)+1 \,\, r < s} F_u = 
\left(\sum_{v=w\,(s',r): \,\,
\ell(v)=\ell(w)+1 \,\, s' < r} F_v\right)\ \ (+F_x)
\end{equation}
where the last term with $x = (1 \times w)(1,r)$ is only present if $\ell(x) = \ell(w)+1$.

One obtains another proof of the Schur positivity of $F_u$ as follows.  Let $r$ be the last descent of $u$, and let $k$ be the largest index such that $u(r) > u(k)$.  Set $w = u(r,k)$.  Then the left hand side of \eqref{eq:transition} has only one term $F_u$.  Recursively repeating this procedure for the terms $F_v$ on the right hand side one obtains a positive expression for $F_u$ in terms of Schur functions.

\item (Little's bijection)
Little \cite{Lit} described an algorithm to  establish \eqref{eq:transition}, which we formulate in the manner of \cite{LS:affineLittle}.  A {\it $v$-marked nearly reduced word} is a pair $(\i,a)$ where $\i = i_1 i_2 \cdots i_\ell$ is a word with letters in $\Z_{>0}$ and $a$ is an index such that $\j = i_1 i_2 \cdots \hat i_a \cdots i_\ell$ is a reduced word for $v$, where $\hat i_a$ denotes omission.  We say that $(\i,a)$ is a marked nearly reduced word if it is a $v$-marked nearly reduced word for some $v$.  A marked nearly reduced word is a marked reduced word if $\i$ is reduced.

Define the directed {\it Little graph} on marked nearly reduced words, where each vertex has a unique outgoing edge $(\i,a) \to (\i',a')$ as follows: $\i'$ is obtained from $\i$ by changing $i_a$ to $i_a-1$.  If $i_a-1 = 0$, then we also increase every letter in $\i$ by $1$.  If $\i$ is reduced then $a'=a$.  If $\i$ is not reduced then $a'$ is the unique index not equal to $a$ such that $i_1 i_2 \cdots \hat i_{a'} \cdots i_\ell$ is reduced.  (Check that this is well-defined.)

For a marked reduced word $(\i,a)$ such that $\i$ is reduced, the {\it forward Little move} sends $(\i,a)$ to $(\j,b)$ where $(\j,b)$ is the first marked reduced word encountered by traversing the Little graph.  
\begin{example} \label{ex:Little} Beginning with $\i = 2134323$ and $a = 5$ one has
$$
2134{\bf 3}21 \to 21342{\bf 2} 1 \to 213421{\bf 1} \to 324532{\bf 1}.
$$
Note that $\i$ is a reduced word for $u = 53142$ which covers $w = 43152$.  The word $3245321$ is a reduced word for $514263 = (1 \times w)(1,2)$.  \end{example}
Check that if you apply the forward Little move to a $w$-marked reduced word $(\i,a)$ where $\i \in R(u)$ for some $u$ on the left hand side of \eqref{eq:transition}, you will get a ($w$ or $1\times w$)-marked reduced word $(\j,b)$ where $\j \in R(v)$ for some $v$ on the right hand side of \eqref{eq:transition}.  This can then be used to prove \eqref{eq:transition}.

\item (Dual Edelman-Greene equivalence)
Let $R(\infty)$ denote the set of all reduced words of permutations.  We say that $\i,\i' \in R(\infty)$ are dual EG-equivalent if the recording tableaux under EG-insertion are the same: $Q(\i) = Q(\i')$.

\begin{conjecture}
Two reduced words are dual $EG$-equivalent if and only if they are connected by forward and backwards Little moves.
\end{conjecture}

For example, both $2134321$ and $3245321$ of Example \ref{ex:Little} Edelman-Greene insert to give recording tableau $$\tableau[sbY]{1&3&4\\2&5\\6\\7}$$

\item
\label{prob:enumerateEG}
Fix a symmetric group $S_n$.  Is there a formula for the number of EG-tableau of a fixed shape $\la$?  (See also Section \ref{ssec:algebraprob} and compare with formulae for the number of (semi)standard tableaux \cite{EC2}.)

\item There are two common bijections which demonstrate the symmetry of Schur functions: the Bender-Knuth involution \cite{BK}, and the Lascoux-Sch\"{u}tzenberger/crystal operators (see for example \cite{LLT}).

Combine this with Edelman-Greene insertion to obtain an explicit weight-changing bijection on the monomials of a Stanley symmetric function, which exhibits the symmetry of a Stanley symmetric function.  Compare with Stanley's original bijection \cite{Sta}.

\item (Jeu-de-taquin for reduced words)
There is a theory of Jeu-de-taquin for skew EG-tableaux due to Thomas and Yong \cite{TY,TY1}, where for example one possible slide is
$$
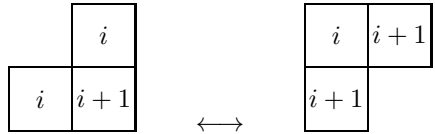
\tableau[mbY]{\bl & i \\ i & i+1} \qquad \longleftrightarrow \qquad \tableau[mbY]{i& i+1 \\ i+1}
$$
\end{enumerate}

\section{Affine Stanley symmetric functions}
\label{sec:affine}
\subsection{Affine symmetric group}
For basic facts concerning the affine symmetric group, we refer the reader to
\cite{BB}.

Let $n > 2$ be a positive integer.  Let $\tS_n$ denote the affine symmetric group with simple
generators $s_0,s_1,\ldots,s_{n-1}$ satisfying the relations
\begin{align*}
s_i^2 &= 1 &\mbox{for all $i$} \\
s_i s_{i+1} s_i &= s_{i+1} s_i s_{i+1} &  \mbox{for all $i$} \\
  s_is_j &= s_j s_i &\mbox{for $|i-j| \geq 2$}.
\end{align*}
Here and elsewhere, the indices will be taken modulo $n$ without
further mention.  The length $\ell(w)$ and reduced words $R(w)$ for affine permutations $w \in \tS_n$ are defined in an analogous manner to Section \ref{sec:perm}.  The symmetric group $S_n$ embeds in $\tS_n$ as the subgroup
generated by $s_1,s_2, \ldots, s_{n-1}$.  

One may realize $\tS_n$ as the set of all
bijections $w:\Z\rightarrow\Z$ such that $w(i+n)=w(i)+n$ for all
$i$ and $\sum_{i=1}^n w(i) = \sum_{i=1}^n i$. In this realization,
to specify an element $w \in \tS_n$ it suffices to give the
``window" $[w(1),w(2),\dotsc,w(n)]$.  The product $w\cdot v$ of
two affine permutations is then the composed bijection $w \circ v:
\Z \rightarrow \Z$.  Thus $ws_i$ is obtained from $w$ by swapping
the values of $w(i+kn)$ and $w(i+kn+1)$ for every $k \in \Z$.  An affine permutation $w \in \tS_n$ is Grassmannian if $w(1) < w(2) < \cdots < w(n)$.  For example, the affine Grassmannian permutation $[-2,2,6] \in \tS_3$ has reduced words $2120$ and $1210$.

\subsection{Definition}
A word $a_1a_2 \cdots a_\ell$ with letters in $\Z/n\Z$ is called {\it cyclically decreasing} if (1) each letter occurs at most once, and (2) whenever $i$ and $i+1$ both occur in the word, $i+1$ precedes $i$.

An affine permutation $w \in \tS_n$ is called cyclically decreasing if it has a cyclically decreasing reduced word.  Note that such a reduced word may not be unique.  

\begin{lemma}
There is a bijection between strict subsets of $\Z/n\Z$ and cyclically decreasing affine permutations $w \in \tS_n$, sending a subset $S$ to the unique cyclically decreasing affine permutation which has reduced word using exactly the simple generators $\{s_i \mid i \in S\}$.
\end{lemma}

We define cyclically decreasing factorizations of $w \in \tS_n$ in the same way as decreasing factorizations in $S_n$.

\begin{definition}
Let $w \in \tS_n$.  The affine Stanley symmetric function $\tF_w$ is given by
$$
\tF_w = \sum_{w = v_1v_2\cdots v_r} x_1^{\ell(v_1)} x_2^{\ell(v_2)} \cdots x_r^{\ell(v_r)}
$$
where the summation is over cyclically decreasing factorizations of $w$.
\end{definition}

\begin{theorem}[\cite{Lam:affStan}]\label{thm:afstsymm}
The generating function $\tF_w$ is a symmetric function.
\end{theorem}

Theorem \ref{thm:afstsymm} can be proved directly, as was done in \cite{Lam:affStan}.  We shall establish Theorem \ref{thm:afstsymm} using the technology of the affine nilHecke algebra in Sections \ref{sec:affinenilHecke}-\ref{sec:affineFS}.  Some immediate observations:
\begin{enumerate}
\item
$\tF_w$ is a homogeneous of degree $\ell(w)$.
\item
If $w \in S_n$, then a cyclically decreasing factorization of $w$ is just a decreasing factorization of $w$, so $\tF_w = F_w$.
\item
The coefficient of $x_1 x_2 \cdots x_{\ell(w)}$ in $\tF_w$ is equal to $|R(w)|$.
\end{enumerate}

\begin{example}\label{ex:21202}
Consider the affine permutation $w = s_2 s_1 s_2 s_0 s_2$.  The reduced words are $R(w) = \{21202, 12102, 21020\}$.  The other cyclically decreasing factorizations are $$\os{21}\os{2}\os{0}\overline{2}, \os{2}\os{1}\os{2}\overline{02},
\os{1}\os{21}\os{0}\overline{2}, \os{1}\os{2}\os{1}\overline{02},\os{1}\os{2}\os{10}\overline{2},\os{21}\os{0}\os{2}\overline{0},\os{2}\os{1}\os{02}\overline{0},\os{2}\os{10}\os{2}\overline{0}$$
$$\os{21}\os{2}\overline{02}, \os{1}\os{21}\overline{02},\os{21}\os{02}\overline{0}$$
Thus
$$
\tF_w = m_{221}+2m_{2111}+3m_{11111}.
$$
\end{example}
\subsection{Codes}
Let $w \in \tS_n$.  The {\it code} $c(w)$ is a vector $c(w) =
(c_1,c_2,\ldots,c_{n}) \in \Z_{\geq 0}^n - \Z_{>0}^n$ of non-negative entries
with at least one 0.  The entries are given by $c_i = \#\{j \in
\Z \, \mid \, j > i \,\, \text{and} \,\, w(j) < w(i)\}$.

It is shown in \cite{BB} that there is a bijection between codes and
affine permutations and that $\ell(w) = |c(w)| := \sum_{i=1}^n c_i$.  We define $\la(w)$ as for usual permutations (see Section \ref{sec:combin}).  For example, for $w = s_2s_0s_1s_2s_1s_0 = [ -4,3,7] \in \tS_3$, one has $c(w^{-1})=(5,1,0)$ and $\la=(2,1,1,1,1)$.

Let $\Bo^n$ denote the set of partitions $\la$ satisfying $\la_1 < n$, called the set of $(n-1)$-bounded partitions.

\begin{lemma}[{\cite{BB}}]
The map $w \mapsto \la(w)$ is a bijection between $\tS_n^0$ and $\Bo^n$.
\end{lemma}

The analogue of Proposition \ref{prop:dominant} has a similar proof.
\begin{prop}[{\cite{Lam:affStan}}] \label{prop:affdominant} Let $w \in \tS_n$.
\begin{enumerate}
\item
Suppose $[m_\la]F_w \neq 0$.  Then $\la \prec \la(w)$.
\item
$[m_{\la(w)}]F_w = 1$.
\end{enumerate}
\end{prop}

\subsection{$\La_{(n)}$ and $\La^{(n)}$}
Let $\La_{(n)} \subset \La$ be the subalgebra generated by $h_1,h_2,\ldots,h_{n-1}$, and let $\La^{(n)}:= \La/I_{(n)}$ where $I_{(n)}$ is the ideal generated by $m_\mu$ for $\mu \notin \Bo^n$.  A basis for $\La_{(n)}$ is given by $\{h_\la \mid \la \in \Bo^n\}$.  A basis for $\La^{(n)}$ is given by $\{m_\la \mid \la \in \Bo^n\}$.

The ring of symmetric functions $\La$ is a Hopf algebra, with coproduct 
given by $\Delta(h_k) = \sum_{j=0}^k h_j \otimes h_{k-j}$.  Equivalently, the coproduct of $f(x_1,x_2,\ldots) \in \La$ can be obtained by writing $f(x_1,x_2,\ldots,y_1,y_2,\ldots)$ in the form $\sum_i f_i(x_1,x_2,\ldots) \otimes g_i(y_1,y_2,\ldots)$ where $f_i$ and $g_i$ are symmetric in $x$'s and $y$'s respectively.  Then $\Delta(f)= \sum_i f_i \otimes g_i$.

The ring $\La$ is self Hopf-dual under the Hall inner product.  That is, one has $\ip{\Delta f, g\otimes h}  = \ip{f,gh}$ for $f,g,h\in \La$.  Here the Hall inner product is extended to $\La \otimes \La$ in the obvious way.  The rings $\La_{(n)}$ and $\La^{(n)}$ are in fact Hopf algebras, which are dual to each other under the same inner product.  We refer the reader to \cite{Mac} for further details.

\subsection{Affine Schur functions}
Stanley symmetric functions expand positively in terms of the basis of Schur functions (Corollary \ref{cor:stpos}).  We now describe the analogue of Schur functions for the affine setting.

For $\la \in \Bo^n$, we let $\tF_\la:=\tF_w$ where $w \in \tS_n^0$ is the unique affine Grassmannian permutation with $\la(w) = \la$.  These functions $\tF_\la$ are called {\it affine Schur functions} (or dual $k$-Schur functions, or weak Schur functions).

\begin{theorem}[\cite{LM:ktab,Lam:affStan}]\label{thm:affineSchurbasis}
The affine Schur functions $\{\tF_\la\mid \la \in \Bo^m\}$ form a basis of $\La^{(n)}$.
\end{theorem}
\begin{proof}
By Proposition \ref{prop:affdominant}, the leading monomial term of $\tF_\la$ is $m_\la$.  Thus $\{\tF_\la\mid \la \in \Bo^m\}$ is triangular with resepect to the basis $\{m_\la\mid \la \in \Bo^m\}$, so that it is also a basis.
\end{proof}

We let $\{s^{(k)}_{\la}\} \subset \La_{(n)}$ denote the dual basis to $\tF_\la$.  These are the (ungraded) {\it $k$-Schur functions}, where $k = n-1$.  It turns out that the $k$-Schur functions are Schur positive.  However, affine Stanley symmetric functions are not.  Instead, one has:

\begin{theorem}[\cite{Lam:Schub}]\label{thm:affstpos}
The affine Stanley symmetric functions $\tF_w$ expand positively in terms of the affine Schur functions $\tF_\la$.
\end{theorem}

Theorem \ref{thm:affstpos} was established using geometric methods.  See Section \ref{sec:geom} and \cite{Lam:Schub}.  It is an open problem to give a combinatorial interpretation of the affine Stanley coefficients, expressing affine Stanley symmetric functions in terms of affine Schur functions.

\subsection{Example: The case of $\tS_3$}
To illustrate Theorem \ref{thm:affineSchurbasis}, we completely describe the affine Schur functions for $\tS_3$.  

\begin{prop}\label{prop:redts3}
Let $w \in \tS_n$ be the affine Grassmannian permutation corresponding to the partition $(2^a1^b)$.  Then $|R(w)| = \binom{\lfloor b/2+a\rfloor}{a}$.
\end{prop}
\begin{prop}\label{prop:ts3}
The affine Schur function $\tF_{2^a,1^b}$ is given by
$$
\tF_{2^a,1^b} = \sum_{j = 0}^a \binom{\lfloor b/2+a-j\rfloor}{a-j} m_{2^j1^{b+2a-2j}}.
$$
The $k$-Schur function $s^{(2)}_{2^a,1^b}$ is given by
$$
s^{(2)}_{2^a,1^b}=h_2^a e_2^{\lfloor b/2 \rfloor} h_1^{b- 2\lfloor b/2 \rfloor}.
$$
\end{prop}

\begin{example}
For $w = 1210$, we have $a = 1$ and $b = 2$.  Thus $R(w) = \{1210,2120\}$ has cardinality $\binom{2}{1} = 2$, and $\tF_{2,1^2} = m_{211}+2m_{1111}$.
\end{example}
\begin{example}
The affine Stanley symmetric function of Example \ref{ex:21202} expands as
$$
\tF_w = \tF_{2^2,1}+\tF_{2,1^3}+\tF_{1^5}
$$
agreeing with Theorem \ref{thm:affstpos}.
\end{example}

\subsection{Exercises and problems}
\begin{enumerate}
\item (Coproduct formula \cite{Lam:affStan})
Show that $\Delta \tF_w = \sum_{uv = w: \ell(w) = \ell(u)+\ell(v)} \tF_u \otimes \tF_v$.
\item (321-avoiding affine permutations \cite{Lam:affStan})
Extend the results in Section \ref{ssec:combinprob} on 321-avoiding permutations to the affine case.  
\item (Affine vexillary permutations)
For which $w \in \tS_n$ is $\tF_w$ equal to an affine Schur function $\tF_\la$?  See the discussion of vexillary permutations in Section \ref{ssec:EGprob} and also \cite[Problem 1]{Lam:affStan}.
\item \label{prob:limit} ($n \to \infty$ limit)
Show that for a fixed partition $\la$, we have $\lim_{n \to \infty} \tF^{(n)}_\la = s_\la$, where $\tF^{(n)}$ denotes the affine Schur function for $\tS_n$.
\item
Extend Proposition \ref{prop:ts3} to all affine Stanley symmetric functions in $\tS_3$, and thus give a formula for the affine Stanley coefficients.
\item
Is there an affine analogue of the fundamental quasi-symmetric functions?  For example, one might ask that affine Stanley symmetric functions expand positively in terms of such a family of quasi-symmetric functions.  Affine Stanley symmetric functions do not in general expand positively in terms of fundamental quasi-symmetric functions (see \cite[Theorem 5.7]{McN}).
\item
Find closed formulae for numbers of reduced words in the affine symmetric groups $\tS_n$, $n > 3$, extending Proposition \ref{prop:redts3}.  Are there formulae similar to the determinantal formula, or hook-length formula for the number of standard Young tableaux?
\item ($n$-cores)
A skew shape $\la/\mu$ is a $n$-ribbon if it is connected, contains $n$ squares, and does not contain a $2\times 2$ square.
An $n$-core $\la$ is a partition such that there does not exist $\mu$ so that $\la/\mu$ is a $n$-ribbon.
There is a bijection between the set of $n$-cores and the affine Grassmannian permutaitons $\tS_n^0$.  Affine Schur functions can be described in terms of tableaux on $n$-cores, called $k$-tableau \cite{LM:core} (or weak tableau in \cite{LLMS}).

\item (Cylindric Schur functions \cite{Pos,McN})
Let $C(k,n)$ denote the set of lattice paths $p$ in $\Z^2$ where every step either goes up or to the right, and which is invariant under the translation $(x,y) \mapsto (x+n-k,y+k)$.  Such lattice paths can be thought of as the boundary of an infinite periodic Young diagram, or equivalently of a Young diagram on a cylinder.  We write $p \subset q$ if $p$ lies completely to the left of $q$.  A cylindric skew semistandard tableau is a sequence $p_0 \subset p_1 \subset \cdots \subset p_k$ of $p_i \in C(k,n)$ where the region between $p_i$ and $p_{i+1}$ does not contain two squares in the same column.  One obtains \cite{Pos} a natural notion of a cylindric (skew) Schur function.  Show that every cylindric Schur function is an affine Stanley symmetric function, and every affine Stanley symmetric function of a 321-avoiding permutation is a cylindric Schur function (\cite{Lam:affStan}).

\item (Kashiwara-Shimozono affine Grothendieck polynomials)
The usual Stanley symmetric functions can be expressed as stable limits of Schubert polynomials \cite{BJS}.  What is the relationship between affine Stanley symmetric functions and the affine Grothendieck polynomials of Kashiwara and Shimozono \cite{KS}?
\item
Is there a good notion of Coxeter-Knuth equivalence for reduced words of affine permutations?  This may have an application to the affine Schur positivity of affine Stanley symmetric functions (Theorem \ref{thm:affstpos}).  See also Section\ref{ssec:EGprob} (\ref{rem:CK}).
\item (Affine Little bijection \cite{LS:affineLittle})
There is an affine analogue of Little's bijection (Section \ref{ssec:EGprob}) developed in \cite{LS:affineLittle}.  It gives a combinatorial proof of the affine analogue of the transition formula \eqref{eq:transition}.  Can the affine Little bijection, or the affine transition formula lead to a proof of Theorem \ref{thm:affstpos}?  Can one define a notion of dual EG-equivalence using the affine Little bijection?
\item (Branching positivity \cite{LLMS2,Lam:ASP}) Let $\tF_\la^{(n)}$ denote the affine Schur functions for $\tS_n$.  Then $\tF_\la^{(n+1)}$ expands positively in terms of $\tF_\mu^{(n)}$ modulo the ideal in symmetric functions generated by $m_\nu$ with $\nu_1 \geq n$.  Deduce using (\ref{prob:limit}) that $k$-Schur functions are Schur positive.
\end{enumerate}

\section{Root systems and Weyl groups}
\label{sec:Weyl}
In this section, we let $\fW$ be a finite Weyl group and $\aW$ denote the corresponding affine Weyl group.
We shall assume basic familiarity with Weyl groups, root systems, and weights \cite{Hum,Kac}.

\subsection{Notation for root systems and Weyl groups}
Let $A = (a_{ij})_{i,j \in \aI}$ denote an affine Cartan matrix, where $\aI = \fI \cup\{0\}$, so that $(a_{ij})_{i,j \in \fI}$ is the corresponding finite Cartan matrix.  For example, for type $\tilde A_{n-1}$ (corresponding to $\tS_n$) and $n > 2$ we have $\aI = \Z/n\Z$ and 
$$
a_{ij} = \begin{cases} 2 &\mbox{if $i = j$} \\
-1 &\mbox{if $j = i\pm 1$} \\
0 & \mbox{otherwise.}
\end{cases}
$$

The affine Weyl group $\aW$ is generated by  involutions $\{s_i \mid i \in \aI\}$ satisfying the relations $(s_is_j)^{m_{ij}} = \id$, where for $i \neq j$, one defines $m_{ij}$ to be $2,3,4,6,\infty$ according as $a_{ij}a_{ji}$ is $0, 1, 2, 3, \ge4$.  The finite Weyl group $\fW$ is generated by $\{s_i \mid i \in \fI\}$.  For the symmetric group $W = S_n$, we have $I = [n-1]$, $m_{i,i+1} = 3$, and $m_{ij} = 2$ for $|i-j| \geq 2$.

Let $R$ be the root system
for $W$.  Let $R^+$ denote the positive roots, $R^-$ denote the
negative roots and $\{ \alpha_i \mid i \in \fI\}$ denote the simple
roots.   Let $\theta$ denote the highest root of $R^+$.  Let $\rho = \frac{1}{2}\sum_{\alpha \in R^+} \alpha$ denote the half sum of positive roots.  Also let $\{\alpha_i^\vee \mid i \in \fI\}$ denote the simple coroots.  

We write $R_\af$ and $R^+_\af$ for the affine root system, and positive affine roots.  The positive simple affine roots (resp. coroots) are $\{\alpha_i \mid i \in \aI\}$ (resp. $\{\alpha^\vee_i \mid i \in \aI\}$).  The null root $\delta$ is given by $\delta = \alpha_0 + \theta$.  A root $\alpha$ is {\it real} if it is in the $\aW$-orbit of the simple affine roots, and {\it imaginary} otherwise.  The imaginary roots are exactly $\{k\delta \mid k \in \Z \setminus \{0\}\}$.  Every real affine root is of the form $\alpha + k\delta$, where $\alpha \in R$.   The root $\alpha+k\delta$ is positive if $k > 0$, or if $k = 0$ and $\alpha \in R^+$.  

Let $Q = \oplus_{i \in \fI} \Z \cdot \alpha_i$ denote the root lattice and let $Q^\vee = \oplus_{i \in I} \Z \cdot \alpha_i^\vee$ denote the co-root
lattice.  Let $P$ and $P^\vee$ denote the weight lattice and
co-weight lattice respectively.  Thus $Q \subset P$ and $Q^\vee \subset P^\vee$.  We also have a map $Q_\af = \oplus_{i \in \aI} \Z \cdot \alpha_i \to P$ given by sending $\alpha_0$ to $-\theta$ (or equivalently, by sending $\delta$ to $0$).  Let $\ip{.,.}$ denote
the pairing between $P$ and $P^\vee$.  In particular, one requires that $\ip{\al_i^\vee,\al_j}=a_{ij}$.  

The Weyl group acts on weights via $s_i \cdot \la = \la - \ip{\alpha_i^\vee,\la}\alpha_i$ (and via the same formula on $Q$ or $Q_\af$), and on coweights via $s_i \cdot \mu = \mu - \ip{\mu,\alpha_i}\alpha^\vee_i$ (and via the same formula on $Q^\vee$).  For a real root $\alpha$ (resp. coroot $\alpha^\vee$), we let $s_\alpha$ 
(resp. $s_{\alpha^\vee}$) denote the corresponding reflection, defined by $s_\alpha = wr_i w^{-1}$ if $\alpha = w \cdot \alpha_i$.  The reflection $s_\alpha$ acts on weights by $s_\alpha \cdot \la = \la - \ip{\alpha^\vee,\la}\alpha$.

\begin{example}
Suppose $W = S_n$ and $\aW = \tS_n$.  We have positive simple roots $\alpha_1,\alpha_2,\ldots,\alpha_{n-1}$ and an affine simple root $\alpha_0$.
The finite positive roots are $R^+=\{ \alpha_{i,j} := \alpha_i + \alpha_{i+1} \cdots + \alpha_{j-1} \mid 1 \leq i < j \leq n\}$.  The reflection $s_{\alpha_{i,j}}$ is the transposition $(i,j)$.  The highest root is $\theta = \alpha_1 + \cdots + \alpha_{n-1}$.  The affine positive roots are $R^+_\af = \{\alpha_{i,j} \mid i < j \}$, where for simple roots the index is taken modulo $n$.  Note that one has 
$\alpha_{i,j} = \alpha_{i+n,j+n}$.  The imaginary roots are of the form $\alpha_{i,i+kn}$.  For a real root $\alpha_{i,j}$, the reflection $s_{\alpha_{i,j}}$ is the affine transposition $(i,j)$.

The weight lattice can be taken to be $P = \Z^n/(1,1,\ldots,1)$, and the coweight lattice to be $P^\vee = \{(x_1,x_2,\ldots,x_n) \in \Z^n \mid \sum_i x_i = 0\}$.  The roots and coroots are then $\alpha_{i,j} = e_i-e_j = \alpha^\vee_{i,j}$ (though the former is in $P$ and the latter is in $P^\vee$).  The inner product $P^\vee \times P \to \Z$ is induced by the obvious one on $\Z^n$.
\end{example}

\subsection{Affine Weyl group and translations}
The affine Weyl group can be expressed as the semi-direct product $\aW= W \ltimes Q^\vee$, as follows.  For each $\la \in Q^\vee$, one has a translation element $t_\la \in \aW$.  Translations are multiplicative, so that $t_\lambda \cdot t_\mu$ =
$t_{\lambda+\mu}$.  We also have the conjugation formula $w\, t_\lambda
w^{-1} = t_{w \cdot \lambda}$ for $w \in W$ and $\la \in Q^\vee$.
Let
$s_0$ denote the additional simple generator of $\aW$.  Then translation elements are related to the simple generators via the formula$$s_0 = s_{\theta^\vee} t_{-\theta^\vee}.$$ 

\begin{example}
For $\aW = \tS_n$, and $\la = (\la_1,\la_2,\ldots,\la_n) \in Q^\vee$, we have
$$
t_\la= [1+n\la_1,2+n\la_2,\ldots,n+n\la_n].
$$
Thus $t_{-\theta^\vee} = [1-n,2,\ldots,n-1,2n]$ and $s_0=s_{\theta^\vee}t_{-\theta^\vee}$ is the equality
$$
[0,2,\ldots,n-1,n+1]=[n,2,\ldots,n-1,1]\cdot[1-n,2,\ldots,n-1,2n].
$$
\end{example}

The element $wt_\la$ acts on $\mu \in P$ via
\begin{equation}\label{eq:affinewaction}
wt_\la \cdot \mu = w \cdot \mu.
\end{equation}
In other words, the translations act trivially on the {\it finite} weight lattice.  This action is called the {\it level zero action}.  

Let $\ell: \aW \rightarrow \Z_{\geq 0}$ denote the length function of
$\aW$.  Thus $\ell(w)$ is the length of the shortest reduced factorization of $w$.  

\begin{exercise}
For $w\, t_\la \in \aW$, we have
\begin{equation}
\label{eq:length} \ell(w\,t_\lambda) = \sum_{\alpha \in R^+} |
\ip{\lambda,\alpha} + \chi(w\cdot\alpha)|,
\end{equation}
where $\chi(\alpha) = 0$ if $\alpha \in R^+$ and $\chi(\alpha) = 1$
otherwise.
\end{exercise}

A coweight $\la$ is {\it dominant} (resp. {\it anti-dominant}) if $\ip{\la,\alpha} \geq 0$ (resp. $\leq 0$) for every root $\alpha \in R^+$.

\begin{exercise}
Suppose $\la \in Q^\vee$ is dominant.  Then $\ell(t_{w\la}) = 2\ip{\la,\rho}$.
\end{exercise}

Let $\aW^0$ denote the minimal length coset representatives of
$\aW/W$, which we call {\it Grassmannian} elements. There is a
natural bijection between $\aW^0$ and $Q^\vee$: each coset $\aW/W$
contains one element from each set.  

\begin{exercise}
We have $\aW^0 \cap Q^\vee = Q^-$, the elements of the coroot lattice which are anti-dominant.
\end{exercise}

In fact an element $w t_\lambda$ lies in $\aW^0$ if and only if $t_\lambda \in Q^-$ and $w \in W^\lambda$ where
$W^\lambda$ is the set of minimal length representatives of $W/W_\lambda$ and $W_\lambda$ is the stabilizer subgroup of
$\lambda$.

\section{NilCoxeter algebra and Fomin-Stanley construction}
\label{sec:algebra}
Let $W$ be a Weyl group and $\aW$ be the corresponding affine Weyl group.
\subsection{The nilCoxeter algebra}
The {\it nilCoxeter algebra} $\A_0$ is the algebra over $\Z$ generated by $\{A_i \mid i \in I\}$ with relations
\begin{align*}
A_i^2 &= 0 \\ 
(A_iA_j)^b &= (A_jA_i)^b & \text{if } (s_is_j)^b = (s_js_i)^b \\ A_j(A_iA_j)^b &= (A_iA_j)^bA_i &\text{if } s_j(s_is_j)^b = (s_is_j)^bs_i
\end{align*}
The algebra $\A_0$ is graded, where $A_i$ is given degree $1$.

The corresponding algebra for the affine Weyl group will be denoted $\aNC$.

\begin{proposition}
The nilCoxeter algebra $\A_0$ has basis $\{A_w \mid w \in W\}$, where $A_w = A_{i_1}A_{i_2} \cdots A_{i_\ell}$ for any reduced word $i_1 i_2 \cdots i_\ell$ of $w$.  The mulitplication is given by
$$
A_w \,A_v = \begin{cases} A_{wv} & \mbox{if $\ell(wv) = \ell(w) + \ell(v)$} \\
0 & \mbox{otherwise.}
\end{cases}
$$
\end{proposition}

\subsection{Fomin and Stanley's construction}
Suppose $W = S_n$.  We describe the construction of Stanley symmetric functions of Fomin and Stanley \cite{FS}.  Define
$$
\h_k = \sum_{w \text{ decreasing: } \ell(w) = k} A_w.
$$
For example, when $n = 4$, we have
\begin{align*}
\h_0 &= \id \\
\h_1&=A_1 + A_2 + A_3 \\
\h_2&=A_{21} + A_{31} + A_{32} \\
\h_3&=A_{321}
\end{align*}
where $A_{i_1i_2\cdots i_\ell}:= A_{i_1}A_{i_2} \cdots A_{i_\ell}$.  

\begin{lemma}[\cite{FS}]\label{lem:hproduct}
The generating function $\h(t) = \sum_k \h_k t^k$ has the product expansion
$$
\h(t) = (1+t\,A_{n-1})(1+t\,A_{n-2}) \cdots (1+t\,A_1).
$$
\end{lemma}

\begin{definition}[NilCoxeter algebra] \label{def:stalg}
The Stanley symmetric function $F_w$ is the coefficient of $A_w$ in the product
$$
\h(x_1)\h(x_2)\cdots
$$
\end{definition}

\begin{lemma}[\cite{FS}]\label{lem:finstcommute}
We have 
$$
\h(x) \h(y) = \h(x) \h(y).
$$
Thus for every $k, l$ we have
$$
\h_k \h_l = \h_l \h_k.
$$
\end{lemma}
\begin{proof}
We observe that $(1+xA_i)$ and $(1+yA_j)$ commute whenever $|i-j| \geq 2$ and that
$$
(1+ xA_{i+1})(1+ xA_{i})(1+yA_{i+1})
=(1+yA_{i+1})(1+ yA_{i})(1+ xA_{i+1})(1-yA_i)(1+xA_i).
$$
Assuming by induction that the result is true for $S_{n-1}$ we calculate
\begin{align*}
&(1+xA_{n-1})\cdots (1+xA_1)(1+yA_{n-1})\cdots (1+yA_1) \\
&= \left[(1+yA_{n-1})(1+yA_{n-2})(1+xA_{n-1})(1-yA_{n-2})(1+xA_{n-2})\right]\\&(1+xA_{n-3})\cdots (1+xA_1)(1+yA_{n-2})\cdots (1+yA_1) \\
&= (1+yA_{n-1})(1+yA_{n-2})(1+xA_{n-1})(1-yA_{n-2})\left[(1+yA_{n-2})\cdots (1+yA_1)\right]\\& \left[ (1+xA_{n-2})\cdots (1+xA_1) \right]\\
&=(1+yA_{n-1})\cdots (1+yA_1)(1+xA_{n-1})\cdots (1+xA_1).
\end{align*}
\end{proof}

\begin{proof}[Proof of Theorem \ref{thm:stsymm}]
Follows immediately from Definition \ref{def:stalg} and Lemma \ref{lem:finstcommute}.
\end{proof}

The following corollary of Lemma \ref{lem:finstcommute} suggests a way to generalize these constructions to other finite and affine Weyl groups.

\begin{corollary}\label{cor:finFS}
The elements $\h_k$ generate a commutative subalgebra of $\A_0$.
\end{corollary}

We call the subalgebra of Corollary \ref{cor:finFS} the {\it Fomin-Stanley} sublagebra of $\A_0$, and denote it by $\B$.  As we shall explain, the combinatorics of Stanley symmetric functions    is captured by the algebra $\B$, and the information can be extracted by ``picking a basis''.

\subsection{A conjecture}
We take $W$ to be an arbitrary Weyl group.  For basic facts concerning the exponents of $W$, we refer the reader to \cite{Hum}.
The following conjecture was made by the author and Alex Postnikov.  Let $(R^+, \prec)$ denote the partial order on the positive roots of $W$ given by $\alpha \prec \beta$ if $\beta - \alpha$ is a positive sum of simple roots, and let $J(R^+,\prec)$ denote the set of upper order ideals of $(R^+,\prec)$.

\begin{conjecture}\label{conj:LP}
The (finite) nilCoxeter algebra $\A_0$ contains a graded commutative subalgebra $\B'$ satisfying:
\begin{enumerate}
\item
Over the rationals, the algebra $\B' \otimes_\Z \Q$ is generated by homogeneous elements $\h_{i_1}, \h_{i_2}, \ldots, \h_{i_r} \in \A_0$ with degrees $\deg(\h_{i_j}) = i_j$ given by the exponents of $W$.
\item
The Hilbert series $P(t)$ of $\B'$ is given by 
$$
P(t) = \sum_{I \in J(R^+,\prec)} t^{|I|}
$$
In particular, the dimension of $\B'$ is a the generalized Catalan number for $W$ (see for example \cite{FR}).
\item 
The set $\B'$ has a homogeneous basis $\{b_I \mid I \in J(R^+,\prec)\}$ consisting of elements which are nonnegative linear combinations of the $A_w$.
\item
The structure constants of the basis $\{b_I\}$ are positive.
\end{enumerate}
\end{conjecture}

In the sequel we shall give an explicit construction of a commutative subalgebra $\B$ and provide evidence that it satisfies Conjecture \ref{conj:LP}.  

Suppose $W = S_n$.  We show that $\B'=\B$ satisfies Conjecture \ref{conj:LP}.  The upper order ideals $I$ of $(R^+,\prec)$ are naturally in bijection with Young diagrams fitting inside the staircase $\delta_{n-1}$.  For each partition $\la$ we define the noncommutative Schur functions $\s_\la \in \B$, following Fomin and Greene \cite{FG}, by writing $s_\la$ as a polynomial in the $h_i$, and then replacing $h_i$ by $\h_i$.  (We set $\h_k = 0$ if $k \geq n$, and $\h_0 = 1$.)  Fomin and Greene show that $\s_\la$ is a nonnegative linear combination of the $A_w$'s (but it will also follow from our use of the Edelman-Greene correspondence below).

\begin{prop}[Lam - Postnikov]\label{prop:LP}
The set $\{\s_\la \mid \la \subseteq \delta_{n-1}\} \subset \B$ is a basis for $\B$.
\end{prop}
\begin{proof}
Let $\ip{.,.}:\A_0 \otimes \A_0 \to \Z$ be the inner product defined by extending bilinearly $\ip{A_w,A_v} = \delta_{wv}$.
Rewriting the definition of Stanley symmetric functions and using the Cauchy identity, one has
\begin{align*}
F_w &= \sum_\la \ip{\h_{\la_1}\h_{\la_2}\cdots, A_w}m_\la\\
&=\sum_\la \ip{\s_\la, A_w}s_\la.
\end{align*}
It follows that the coefficient of $A_w$ in $\s_\la$ is $\alpha_{w\la}$, the coefficient of $s_\la$ in $F_w$.  By Lemma \ref{lem:EGshape}, we have $\s_\la = 0$ unless $\la \subseteq \delta_{n-1}$.  It remains to show that this set of $\s_\la$ are linearly independent.  To demonstrate this, we shall find, for each $\la \in \delta_{n-1}$, some $w \in S_n$ such that $F_w = s_\la + \sum_{\mu \prec \la} \alpha_{w\mu} s_\mu$.

Since $s_\la = m_\la + \sum_{\mu \prec \la} K_{\la\mu} m_\mu$ (where the $K_{\la\mu}$ are the Kostka numbers), by Proposition \ref{prop:dominant} it suffices to show that the permutation $w$ with code $c(w) = \la$ lies in $S_n$.  Let $\la = (\la_1\geq\la_2\geq \ldots \geq\la_{\ell} > 0)$.  Then define $w$ recursively by $w(i) = \min\{j > \la_i \mid j \notin \{w(1),w(2),\ldots,w(i-1)\}$.  Since $\la_i \leq n-i$, we have $w(i) \leq n$.  By construction $w \in S_n$ and has code $c(w) = \la$.
\end{proof}

\begin{example}\label{ex:S4}
Let $W = S_4$.  Write $A_{i_1i_2\cdots i_k}$ for $A_{s_{i_1}s_{i_2}\cdots s_{i_k}}$.

Then we have
\begin{align*}
\begin{array}{ll}
\s_\emptyset = 1 &
\s_1 = A_1 + A_2 + A_3 \\
\s_{11} = A_{12} + A_{23} + A_{13} & \s_2 = A_{21}+A_{32}+A_{31} \\
\s_{111} = A_{123}  & \s_3 = A_{321} \\
\s_{21} = A_{213}+A_{212}+A_{323}+A_{312}&
\s_{211}= A_{1323} + A_{1213} \\ 
\s_{22}= A_{2132} 
&\s_{31}= A_{3231}+A_{3121} \\
\s_{221}=A_{23123} & \s_{311}=A_{32123} \\
\s_{32}=A_{32132}&
\s_{321}=A_{321323}
\end{array}
\end{align*}
\end{example}

\subsection{Exercises and Problems}
\label{ssec:algebraprob}
\begin{enumerate}
\item 
In Example \ref{ex:S4} every $A_w$ occurs exactly once, except for $w = s_1s_3$.  Explain this using Theorem \ref{thm:CK}.
\item (Divided difference operators)
Let $\partial_i$, $1 \leq i \leq n-1$ denote the divided difference operator acting on polynomials in $x_1,x_2,\ldots,x_n$ by
$$
\partial_i \cdot f(x_1,\ldots,x_n) = \frac{f(x_1,x_2,\ldots,x_n) - f(x_1,\ldots,x_{i-1},x_{i+1},x_i,x_{i+2},\ldots,x_n)}{x_i - x_{i+1}}.
$$
Show that $A_i \mapsto \partial_i$ generates an action of the nilCoxeter algebra $\A_0$ for $S_n$ on polynomials.  Is this action faithful?
\item (Center)
What is the center of $\A_0$ and of $\aNC$?
\item
Does an analogue of Conjecture \ref{conj:LP} also hold for finite Coxeter groups which are not Weyl groups?
\item 
How many terms (counted with multiplicity) are there in the elements $\s_\la$ of Proposition \ref{prop:LP}?  This is essentially the same problem as (\ref{prob:enumerateEG}) in Section \ref{ssec:EGprob} (why?).
\end{enumerate}

\section{The affine nilHecke ring}\label{sec:affinenilHecke}
Kostant and Kumar \cite{KK} introduced a nilHecke ring to study the topology of Kac-Moody flag varieties.  Let $W$ be a Weyl group and $\aW$ be the corresponding affine Weyl group.

\subsection{Definition of affine nilHecke ring}
In this section we study the affine nilHecke ring $\afA$ of Kostant and Kumar \cite{KK}.  Kostant and Kumar define the nilHecke ring in the Kac-Moody setting.  The ring $\afA$ below is a ``small torus'' variant of their construction for the affine Kac-Moody case.  

The affine nilHecke ring $\afA$ is the ring with a $1$
given by generators $\{A_i \mid i \in \aI \} \cup \{\lambda
\mid \lambda \in P \}$ and the relations
\begin{align}\label{eq:aila}
A_i \,\lambda &= (s_i \cdot \lambda)\, A_i +
\ip{\alpha_i^\vee,\la}\cdot1 & \mbox{for $\la \in P$,} \\ 
A_i\, A_i &= 0, \\ 
\label{eq:aiaj}
(A_iA_j)^m &= (A_jA_i)^m & \mbox{if $(s_is_j)^m = (s_js_i)^m$.}
\end{align}
where the ``scalars'' $\la \in P$ commute with other
scalars.  Thus $\afA$ is obtained from the affine nilCoxeter algebra $\aNC$ by adding the scalars $P$.  The finite nilHecke ring is the subring $\A$ of $\afA$
generated by $\{A_i \mid i \in \fI \} \cup \{\lambda \mid \lambda \in
P\}$. 

Let $w \in \aW$ and let $w = s_{i_1} \cdots s_{i_l}$ be a reduced
decomposition of $w$.  Then $A_w := A_{i_1} \cdots A_{i_l}$ is a
well defined element of $\afA$, where $A_\id = 1$.

Let $S = \Sym(P)$ denote the symmetric algebra of $P$.  The following 
basic result follows from \cite[Theorem 4.6]{KK}, and can be proved directly from the definitions.

\begin{lemma}\label{lem:Abasis}
The set $\{A_w \mid w \in \aW\}$ is an
$S$-basis of $\afA$.
\end{lemma}

\begin{lemma}\label{lem:whomo}
The map $\aW \mapsto \afA$ given by 
$s_i \mapsto 1 - \alpha_i A_i \in \afA$
is a homomorphism.
\end{lemma}
\begin{proof}
We calculate that
\begin{align*}
s_i^2 &= 1 -2\al_i A_i + \al_i A_i \al_i A_i \\
&= 1 - 2\al_i A_i + \al_i (-\al_iA_i + 2)A_i &\mbox{using \eqref{eq:aila}}\\
&= 1 & \mbox{using $A_i^2 = 0$.}
\end{align*}
If $A_i A_j = A_i A_j$ then $\al_i A_j = A_j \al_i$ so it is easy to see that $s_is_j = s_js_i$.  Suppose $A_i A_j A_i = A_j A_i A_j$.  Then $a_{ij} = a_{ji} = -1$ and we calculate that
\begin{align*}
s_i s_j s_i & = (1-\al_iA_i)(1-\al_j A_j)(1-\al_i A_i)  \\
&=  1 - (2\al_iA_i+\al_jA_j) + (\al_iA_i \al_j A_j +\al_i A_i\al_i A_i + \al_jA_j \al_iA_i)
- \al_iA_i \al_j A_j \al_i A_i \\
&= 1 - (2\al_iA_i+\al_jA_j )+ (2\al_iA_j+2\al_iA_i + 2\al_j A_i)  + \al_i(\al_i+\al_j)A_i A_j + \\ &\;\;\;\;\;\al_j(\al_i+\al_j)A_j A_i - (\al_i A_i + \al_i (\al_i + \al_j) \al_jA_i A_j A_i) \\
&= 1 + (2 \al_i A_j + 2 \al_j A_i -\al_iA_i -\al_j A_j ) + \al_i(\al_i+\al_j)A_i A_j +\\ &\;\;\;\;\;\al_j(\al_i+\al_j)A_j A_i -  \al_i (\al_i + \al_j) \al_j A_i A_j A_i.
\end{align*}
Since the above expression is symmetric in $i$ and $j$, we conclude that $s_is_js_i = s_js_is_j$.
\end{proof}

\begin{exercise}
Complete the proof of Lemma \ref{lem:whomo} for $(A_iA_j)^2= (A_jA_i)^2$ and $(A_iA_j)^3= (A_jA_i)^3$.
\end{exercise}

It follows from Lemma \ref{lem:Abasis} that the map of Lemma \ref{lem:whomo} is an isomorphism onto its image.  Abusing notation, we write $w \in \afA$ for the element in the
nilHecke ring corresponding to $w \in \aW$ under the map of Lemma \ref{lem:whomo}.  Then
$\aW$ is a basis for $\afA \otimes_S \Frac(S)$ over $\Frac(S)$ (not over $S$
since $A_i = \frac{1}{\alpha_i}(1-s_i)$).

\begin{lemma}\label{lem:waction}
Suppose $w \in \aW$ and $s \in S$.  Then $w s = (w\cdot s) w$.
\end{lemma}
\begin{proof}
It suffices to establish this for $w = s_i$, and $s = \la \in P$.  We calculate
\begin{align*}
s_i \la &= (1-\al_i A_i)\la \\
&= \la - \al_i (s_i \cdot \la) A_i - \al_i \ip{\al_i^\vee,\la}\\
&= (s_i \cdot \la) - (s_i\cdot \la) \al_i A_i \\
& = (s_i \cdot \la) s_i.
\end{align*}
\end{proof}


\subsection{Coproduct}
\label{sec:nilHeckecoproduct}
We follow Peterson \cite{Pet} in describing the coproduct of $\afA$.  Kostant and Kumar \cite{KK} take a slightly different approach.
\begin{prop} \label{P:coprodaction} Let $M$ and $N$ be left $\afA$-modules.
Define
\begin{equation*}
  M \otimes_{S} N = (M\otimes_\Z N)/\langle sm\otimes n - m \otimes sn \mid
  \text{$s\in S$, $m\in M$, $n\in N$} \rangle.
\end{equation*}
Then $\afA$ acts on $M \otimes_{S} N$ by
\begin{align*}
s \cdot (m\otimes n) &= sm \otimes n \\
 A_i \cdot (m\otimes n) &= A_i \cdot m \otimes n + m \otimes A_i\cdot
  n - \al_i A_i\cdot m\otimes A_i\cdot n.
\end{align*}
Under this action we have
\begin{equation} \label{E:tensorWeyl}
  w \cdot (m\otimes n) = wm \otimes wn
\end{equation}
for any $w \in \aW$.
\end{prop}
\begin{proof}
By $s_i = 1 - \al_i A_i$, we see that
\begin{align*}
s_i \cdot (m \otimes n) & = m \otimes n - \al_iA_i \cdot m \otimes n - m \otimes \al_iA_i\cdot
  n + \al_i A_i\cdot m\otimes \al_iA_i\cdot n \\
&= s_i \cdot m \otimes s_i \cdot n.
\end{align*}
The formula \eqref{E:tensorWeyl} then follows.  Also
$$
(s_i s) \cdot (m \otimes n) = (s_i s) \cdot m \otimes s_i \cdot n = (s_i \cdot s)s_i \cdot (m \otimes n)
$$
agreeing with Lemma \ref{lem:waction}.
Since $\aW$ forms a basis of $\afA$ over $\Frac(S)$, this shows that one has an action of $\afA \otimes_S \Frac(S)$, assuming that scalars can be extended to $\Frac(S)$.   This would be the case if $M \otimes_S N$ is a free $S$-module, which is the case if we take $M = \afA = N$.  Since $\afA \otimes_S \afA$ is the universal case, it follows that one obtains an action 
of $\afA$ in general.
\end{proof}

Consider the case $M=\afA = N$. By Proposition \ref{P:coprodaction}
there is a left $S$-module homomorphism $\Delta:\afA \to \afA
\otimes_{S} \afA$ defined by $\Delta(a) = a \cdot (1 \otimes 1)$. It satisfies
\begin{align}
  \Delta(q) &= s \otimes 1 &\qquad&\text{for $s\in S$} \\
\label{E:DeltaT}
  \Delta(A_i) &= 1 \otimes A_i + A_i \otimes 1 - \al_i A_i \otimes A_i&\qquad&\text{for $i\in I$.}
\end{align}

Let $a\in\afA$ and $\Delta(a)=\sum_{v,w} a_{v,w} A_v \otimes A_w$ with
$a_{v,w}\in S$. In particular if
$b\in\afA$ and $\Delta(b)=\sum_{v',w'} b_{v',w'} A_{v'} \otimes A_{w'}$ then
\begin{equation}\label{E:Deltamult}
\Delta(ab) =  \Delta(a)\cdot \Delta(b) := \sum_{v,w,v',w'} a_{v,w}
b_{v',w'} A_v A_{v'} \otimes A_wA_{w'}.
\end{equation}

\begin{remark}
We caution that $\afA \otimes _S \afA$ is not a well-defined ring with the obvious multiplication.
\end{remark}

\subsection{Exercises and Problems}
The theory of nilHecke rings in the Kac-Moody setting is well-developed \cite{KK,Kum}.
\begin{enumerate}
\item
The following result is \cite[Proposition 4.30]{KK}.  Let $w \in \aW$ and $\la \in P$. Then
\[
A_w \lambda = (w \cdot \lambda)A_w + \sum_{w\,s_\alpha: \; \ell(ws_\alpha) = \ell(w) - 1}\ip{\al^\vee,\lambda} A_{w\,s_\alpha},
\]
where $\alpha$ is always taken to be a positive root of $\aW$.  The coefficients $\ip{\al^\vee,\la}$ are known as {\it
Chevalley coefficients}.
\item
(Center) What is the center of $\afA$?  (See \cite[Section 9]{Lam:Schub} for related discussion.)
\end{enumerate}

\section{Peterson's centralizer algebras}\label{sec:Pet}
Peterson studied a subalgebra of the affine nilHecke ring in his work \cite{Pet} on the homology of the affine Grassmannian.

\subsection{Peterson algebra and $j$-basis}
The Peterson centralizer subalgebra $\Pet$ is the centralizer $Z_{\afA}(S)$
of the scalars $S$ in the affine nilHecke ring $\afA$.  In this section we establish some basic
properties of this subalgebra.  The results here are unpublished works of Peterson.

\begin{lemma}\label{lem:Ptrans}
Suppose $a \in \afA$.  Write $a = \sum_{w \in \aW} a_w w$, where $a_w \in \Frac(S)$.  Then
$a \in \Pet$ if and only if $a_w = 0$ for all non-translation elements $w \neq t_\la$.
\end{lemma}
\begin{proof}
By Lemma \ref{lem:waction}, we have for $s \in S$
$$
(\sum_w a_w w)s = \sum_w a_w (w\cdot s) w
$$
and so $a \in \Pet$ implies $a_w(w \cdot s) = a_w s$ for each $s$.  But using \eqref{eq:affinewaction}, one sees that every $w \in \aW$ acts non-trivially on $S$ except for the translation elements $t_\la$.  Since $S$ is an integral domain, this implies that $a_w = 0$ for all non-translation elements.
\end{proof}

\begin{lemma}\label{lem:Petcomm}
The subalgebra $\Pet$ is commutative.
\end{lemma}
\begin{proof}
Follows from Lemma \ref{lem:Ptrans} and the fact that the elements $t_\la$ commute, and commute with $S$.
\end{proof}

The following important result is the basis of Peterson's approach to affine Schubert calculus via the affine nilHecke ring.

\begin{theorem}[\cite{Pet,Lam:Schub}]\label{thm:jbasis}
The subalgebra $\Pet$ has a basis $\{j_w \mid w \in \aW^0\}$ where
$$
j_w = A_w + \sum_{w \in \aW - \aW^0} j_w^x A_x
$$
for $j_w^x \in S$.
\end{theorem}
Peterson constructs the basis of Theorem \ref{thm:jbasis} using the geometry of
based loop spaces (see \cite{Pet,Lam:Schub}).  We sketch a purely algebraic proof of this theorem, following the ideas of Lam, Schilling, and Shimozono \cite{LSS}.

\subsection{Sketch proof of Theorem \ref{thm:jbasis}}
Let $\Fun(\aW, S)$ denote the set of functions $\xi: \aW \to S$.  We may think of functions $\xi \in \Fun(\aW,S)$ as functions on $\afA$, by the formula $\xi(\sum_{w \in \aW} a_w \, w) = \sum_{w \in \aW} \xi(w) a_w$.  Note that if $a = \sum_{w \in \aW} a_w \, w \in \afA$, the $a_w$ may lie in $\Frac(S)$ rather than $S$, so that in general $\xi(a) \in \Frac(S)$.  Define 
$$
\aPsi:= \{\xi \in \Fun(\aW,S) \mid \xi(a) \in S \text{ for all $a \in \afA$}\}.
$$
Let 
$$
\aPsi^0 = \{\xi \in \aPsi \mid \xi(w) = \xi(v)  \text{ whenever $wW = vW$}\}.
$$
It follows easily that $\aPsi$ has a basis over $S$ given by $\{\xi^w \mid w \in \aW\}$ where $\xi^w(A_v) = \delta_{vw}$ for every $ v\in \aW$.  Similarly, $\aPsi^0$ has a basis given by  $\{\xi^w_0 \mid w \in \aW^0\}$ where $\xi^w_0(A_v) = \delta_{vw}$ for every $v \in \aW^0$.

The most difficult step is the following statement:
\begin{lem}\label{lem:wrongway}
There is a map $\om: \aPsi \to \aPsi^0$ defined by $\om(\xi)(t_\la) = \xi(t_\la)$.
\end{lem}

In other words, $\om(\xi)$ remembers only the values of $\xi$ on translation elements, and one notes that $\om(\xi^w) = \xi^w_0$ for $w \in \aW^0$.  By Lemma \ref{lem:Ptrans}, $\xi(a) = \om(\xi)(a)$ for $\xi \in \aPsi$ and $a \in \Pet$.  We define $\{j_w \in \Pet \otimes \Frac(S) \mid w \in \aW^0\}$ by the equation
$$
\xi_0^v(j_w) = \delta_{wv}
$$
for $w,v \in \aW^0$.  Such elements $j_w$ exist (and span $\Pet \otimes \Frac(S)$ over $\Frac(S)$) since each $\xi \in \aPsi^0$ is determined by its values on the translations $t_\la$.  The coefficient $j_w^x$ of $A_x$ in $j_w$ is equal to the coefficient of $\xi^w_0$ in $\om(\xi^x)$, which by Lemma \ref{lem:wrongway}, must lie in $S$.  This proves Theorem \ref{thm:jbasis}.

Finally, Lemma \ref{lem:wrongway} follows from the following description of $\aPsi$ and $\aPsi^0$, the latter due to Goresky-Kottwitz-Macpherson \cite[Theorem 9.2]{GKM}.  See \cite{LSS} for an algebraic proof in a slightly more general situation.
\begin{proposition}
Let $\xi \in \Fun(\aW,S)$.  Then $\xi \in \aPsi$ if and only if for each $\alpha \in R$, $w \in \aW$, and each integer $d > 0$ we have
\begin{equation}\label{eq:smallGKM}
\xi(w(1-t_{\alpha^\vee})^{d-1}) \qquad \mbox{is divisible by $\alpha^{d-1}$}
\end{equation}
and
\begin{equation}
\xi(w(1-t_{\alpha^\vee})^{d-1}(1-r_\alpha)) \qquad \mbox{is divisible by $\alpha^d$.}
\end{equation}
Let $\xi \in \Fun(\aW,S)$ satisfy $\xi(w) = \xi(v)$ whenever $wW = vW$.  Then $\xi \in \aPsi^0$ if it satisfies \eqref{eq:smallGKM}.
\end{proposition}
\begin{remark}
The ring $\aPsi$ is studied in detail by Kostant and Kumar \cite{KK} in the Kac-Moody setting.
\end{remark}

\subsection{Exercises and Problems}
\begin{enumerate}
\item
Show that $\Delta$ sends $\Pet$ to $\Pet \otimes \Pet$.  Show that the 
coproduct structure constants of $\Pet$ in the $j$-basis are special cases of the coproduct structure constants of $\afA$ in the
$A_w$-basis.

\item ($j$-basis for translations \cite{Pet, Lam:Schub})
Prove using Theorem \ref{thm:jbasis} that $$j_{t_\la} = \sum_{\mu \in W \cdot \la} A_{t_{\mu}}.$$

\item ($j$-basis is self-describing)
Show that the coefficients $j_w^x$ directly determine the structure constants of the $\{j_w\}$ basis.

\item
Find a formula for $j_{s_i t_\la}$ (see \cite[Proposition 8.5]{LS:QH} for a special case).

\item
Extend the construction of $\Pet$, and Theorem \ref{thm:jbasis} to extended affine Weyl groups.


\item (Generators)
Find generators and relations for $\Pet$.  This does not appear to be known even if type $A$.

\item 
Find general formulae for $j_w$ in terms of $A_x$.  See \cite{LS:QH} for a formula in terms of quantum Schubert polynomials, 
which however is not very explicit.
\end{enumerate}

\section{(Affine) Fomin-Stanley algebras}\label{sec:affineFS}
Let $\phi_0: S \to \Z$ denote the map which sends a polynomial to its constant term.  For example, $\phi_0(3\alpha^2_1\alpha_2 + \alpha_2 +5) = 5$.

\subsection{Commutation definition of affine Fomin-Stanley algebra}
We write $\aNC$ for the affine nilCoxeter algebra.  There is an {\it evaluation at 0} map $\phi_0: \afA \to \aNC$ given by $\phi_0(\sum_w a_w\, A_w) = \sum_w \phi_0(a_w)\,A_w$.  We define the {\it affine Fomin-Stanley subalgebra} to be $\aB = \phi_0(\Pet) \subset \aNC$.  The following results follow from Lemma \ref{lem:Petcomm} and Theorem \ref{thm:jbasis}.

\begin{lemma}
The set $\aB \subset \aNC$ is a commutative subalgebra of $\aNC$.
\end{lemma}

\begin{theorem}[\cite{Lam:Schub}]\label{thm:j0basis}
The algebra $\aB$ has a basis $\{j_w^0 \mid w \in \aW^0\}$ satisfying
$$
j_w^0 = A_w + \sum_{\stackrel{x \notin \aW^0}{\ell(x) = \ell(w)}} j_w^x  A_x
$$
and $j_w^0$ is the unique element in $\aB$ with unique Grassmanian term $A_w$.
\end{theorem}

\begin{prop}[\cite{Lam:Schub}]\label{prop:affStan}
The subalgebra $\aB \subset \aNC$ is given by
$$
\aB = \{a\in \aNC | \phi_0(as) = \phi_0(s)a \text{ for all } s \in S\}.
$$
\end{prop}
Proposition \ref{prop:affStan} is proved in the following exercises (also see \cite[Propositions 5.1,5.3,5.4]{Lam:Schub}).
\begin{exercise} \
\begin{enumerate}
\item
Check that $a \in \aB$ satisfies $\phi_0(as) = \phi_0(s)a$ for all $s \in S$, thus obtaining one inclusion.
\item
Show that if $a = \sum_{w \in W} a_w A_w \in \A_0$ lies in $\aB$, then $a$ is a multiple of $A_{\id}$.  (Hint: the action of $\A_0$ on $S$ via divided difference operators is faithful.  See Section \ref{ssec:algebraprob}.)
\item
Suppose that $a \in \aNC$ satisfies the condition $\phi_0(as) = \phi_0(s)a$ for all $s \in S$.  Use (2) to show that $a$ must contain some affine Grassmannian term $A_w$, $w \in \aW^0$.  Conclude by Theorem \ref{thm:j0basis} that $a \in \aB$.
\end{enumerate}
\end{exercise}

A basic problem is to describe $\aB$ explicitly.  We shall do so in type $A$, following \cite{Lam:Schub} and connecting to affine Stanley symmetric functions.  For other types, see \cite{LSS:C, Pon}.

\subsection{Noncommutative $k$-Schur functions}
In the remainder of this section, we take $W = S_n$ and $\aW = \tS_n$.  We define
$$
\ah_k = \sum_{w \text{ cyclically decreasing: } \ell(w) = k} A_w
$$
for $k = 1,2,\ldots,n-1$.  Introduce an inner product $\ip{.,.}:\aNC \times \aNC \to \Z$ given by extending linearly $\ip{A_w,A_v}=\delta_{wv}$.

\begin{definition}\label{def:afstalgebra}
The affine Stanley symmetric function $\tF_w$ is given by 
$$
\tF_w = \sum_\alpha \ip{h_{\alpha_{\ell}} \cdots h_{\alpha_1}, A_w} x^\alpha
$$
where the sum is over compositions $\alpha = (\alpha_1,\alpha_2,\ldots,\alpha_\ell)$.
\end{definition}

Below we shall show that
\begin{theorem}[\cite{Lam:Schub}]\label{thm:ahcommute}
The elements $\ah_1,\ah_2,\ldots,\ah_{n-1} \in \aNC$ commute.
\end{theorem}
Assuming Theorem \ref{thm:ahcommute}, Theorem \ref{thm:afstsymm} follows.

Define the noncommutative $k$-Schur functions $\s^{(k)}_\la \in \aNC$ by writing the $k$-Schur functions $s^{(k)}_\la$ (see Section \ref{sec:affine}) as a polynomial in $h_i$, and replacing $h_i$ by $\ah_i$.

\begin{proposition}[\cite{Lam:Schub}]\label{prop:noncommCauchy}
Inside an appropriate completion of $\aNC \otimes \La^{(n)}$, we have
$$
\sum_{\alpha} \ah_\alpha x^\alpha = \sum_{\la \in \Bo^n} \s^{(k)}_\la \tF_\la
$$
where the sum on the left hand side is over all compositions.
\end{proposition}
\begin{proof}
Since $\{s^{(k)}_\la\}$ and $\{\tF_\la\}$ are dual bases, we have by standard results in symmetric functions \cite{EC2,Mac} that
$$
\sum_{\alpha} h_\alpha(y) x^\alpha = \sum_{\la \in \Bo^n} s^{(k)}_\la(y) \tF_\la(x)
$$
inside $\La_{(n)} \otimes \La^{(n)}$.  Now take the image of this equation under the map $\La_{(n)} \to \aNC$, given by $h_i \mapsto \ah_i$.
\end{proof}

It follows from Definition \ref{def:afstalgebra} and Proposition \ref{prop:noncommCauchy} that
$$
\tF_w = \sum_{\la} \ip{\s^{(k)}_\la,A_w} \tF_\la.
$$
Thus the coefficient of $A_w$ in $\s^{(k)}_\la$ is equal to the coefficient of $\tF_\la$ in $\tF_w$.  By Theorem \ref{thm:affineSchurbasis}, it follows that
\begin{equation}\label{eq:noncommkSchur}
\s^{(k)}_\la = A_v + \sum_{w \notin \aW^0} \alpha_{w\la} A_w.
\end{equation}
In particular, 
\begin{proposition}[\cite{Lam:Schub}]\label{prop:aB}
The subalgebra of $\aNC$ generated by $\ah_1,\ldots,\ah_{n-1}$ is isomorphic to $\La_{(n)}$, with basis given by $\s^{(k)}_\la$.
\end{proposition}

\subsection{Cyclically decreasing elements}
For convenience, for $I \subsetneq \Z/n\Z$ we define $A_I:=A_w$, where $w$ is the unique cyclically decreasing affine permutation which uses exactly the simple generators in $I$.

\begin{theorem}[\cite{Lam:Schub}]\label{thm:aB}
The affine Fomin-Stanley subalgebra $\aB$ is generated by the elements $\ah_k$, and we have $j_w^0 = \s^{(k)}_\la$ where $w \in \aW^0$ satisfies $\la(w) = \la$.
\end{theorem}

\begin{example}\label{ex:js3}
Let $W = S_3$.  A part of the $j$-basis for $\aB$ is 
\begin{align*}
j^0_{\id}&=1\\
j^0_{s_0} &= A_0+A_1 + A_2 \\
j^0_{s_1s_0}&= A_{10}+A_{21}+A_{02} \\
j^0_{s_2s_0}&=A_{01}+A_{12}+A_{20} \\
j^0_{s_2s_1s_0}&=A_{101}+A_{102}+A_{210}+A_{212}+A_{020}+A_{021} \\
j^0_{s_1s_2s_0}&=A_{101}+A_{201}+A_{012}+A_{212}+A_{020}+A_{120}
\end{align*}
\end{example}

We give a slightly different proof to the one in \cite{Lam:Schub}.
\begin{proof}
We begin by showing that $\ah_k \in \aB$.  We will view $S$ as sitting inside the polynomial ring $\Z[x_1,x_2,\ldots,x_n]$; the commutation relations of $\afA$ can easily be extended to include all such polynomials.  To show that $\ah_k \in \aB$ it suffices to check that $\phi_0(\ah_k \,x_i) = 0$ for each $i$.  But by the $\Z/n\Z$-symmetry of the definition of cyclically decreasing, we may assume $i = 1$.  

We note that 
$$
A_j\, x_i = \begin{cases}
x_{i+1}A_i  + 1& \mbox{$j = i$}\\
x_{i-1}A_i - 1&\mbox{$j = i-1$} \\
x_i A_j & \mbox{otherwise.} 
\end{cases}
$$
Now let $I \subsetneq \Z/n\Z$ be a subset of size $k$.  Then 
$$
\phi_0(A_I x_1) = \begin{cases}
A_{I \setminus\{1\}}& \mbox{$r,r+1,\ldots,n-1,0,1 \in I$ but $r-1 \notin I$} \\
-\sum_{i =r}^{i = 0} A_{I \setminus \{i\}}& \mbox{$r,r+1,\ldots,n-1,0 \in I$ but $r-1,1 \notin I$}\\
0 &\mbox{otherwise.} 
\end{cases}
$$
Given a size $k-1$ subset $J \subset \Z/n\Z$ not containing $1$, we see that the term $A_J$ comes up in two ways: from $\phi_0(A_{J \cup \{1\}}x_1)$ with a positive sign, and from $\phi_0(A_{J \cup \{r'\}}x_1)$ with a negative sign, where $r+1,\ldots,n-1,0$ all lie in $J$.

Thus $\ah_k \in \aB$ for each $k$.  It follows that the $\ah_k$ commute, and by \eqref{eq:noncommkSchur}, it follows that $j_w^0 = \s^{(k)}_\la$ and in particular $\ah_i$ generate $\aB$.
\end{proof}

\subsection{Coproduct}
The map $\Delta: \afA \to \afA \otimes_S \afA$ equips $\Pet$ with the structure of a Hopf-algebra over $S$: to see that $\Delta$ sends $\Pet$ to $\Pet \otimes_S \Pet$, one uses $\Delta(t_\la) = t_\la \otimes t_\la$ and Lemma \ref{lem:Ptrans}.  Applying $\phi_0$, the affine Fomin-Stanley algebra $\aB$ obtains a structure of a Hopf algebra over $\Z$.

\begin{theorem}[\cite{Lam:Schub}]\label{thm:aBHopf}
The map $\La_{(n)} \to \aB$ given by $h_i \mapsto \ah_i$ is an isomorphism of Hopf algebras.
\end{theorem}
By Proposition \ref{prop:aB}, to establish Theorem \ref{thm:aBHopf} it suffices to show that $\Delta(\ah_k) = \sum_{j=0}^k \ah_j \otimes \ah_{k-j}$.  This can be done bijectively, using the definition in Section \ref{sec:nilHeckecoproduct}.

\subsection{Exercises and Problems}
Let $W$ be of arbitrary type.
\begin{enumerate}
\item 
Find a formula for $j_{s_0}^0$.  See \cite[Proposition 2.17]{LS:DGG}.
\item
For $W = S_3$, use Proposition \ref{prop:ts3} to give an explicit formula for the noncommutative $k$-Schur functions.
\item (Dynkin automorphisms)
Let $\omega$ be an automorphism of the corresponding Dynkin diagram.  Then $\omega$ acts on $\aNC$, and it is easy to see that $\omega(\aB) = \aB$.  Is $\aB$ invariant under $\omega$?  (For $W = S_n$ this follows from Theorem \ref{thm:aB}, since the $\ah_k$ are invariant under cyclic symmetry.) \item (Generators)
Find generators and relations for $\aB$, preferably using a subset of the $j$-basis as generators.  See \cite{LSS:C, Pon} for the classical types.  See also the proof of Proposition \ref{prop:gen} below.
\item (Power sums \cite[Corollary 3.7]{BSZ})
Define the noncommutative power sums $\mathbf{p}_k \in \aB$ as the image in $\aB$ of $p_k \in \La_{(n)}$ under the isomorphism of Proposition \ref{prop:aB}.  Find an explicit combinatorial formula for $\mathbf{p}_k$.  (See also \cite{FG}.)
\item
Is there a nice formula for the number of terms in the expression of $j_w^0$ in terms of $\{A_x\}$?  This is an ``affine nilCoxeter'' analogue of asking for the number of terms $s_\la(1,1,\ldots,1)$ in a Schur polynomial.
\item 
Find a combinatorial formula for the coproduct structure constants in the $j$-basis.  These coefficients are known to be positive \cite{Kum}.
\item 
Find a combinatorial formula for the $j$-basis.  It follows from work of Peterson (see also \cite{Lam:Schub, LS:QH}) that the coefficients $j_w^x$ are positive.
\end{enumerate}

\section{Finite Fomin-Stanley subalgebra}
\label{sec:finiteFS}
In this section we return to general type.

There is a linear map $\kappa:\A \to \aNC$ given by 
$$
\kappa(A_w) = \begin{cases} A_w & \mbox{$w \in W$} \\
0 & \mbox{otherwise.}
\end{cases}
$$
The {\it finite Fomin-Stanley algebra} $\B$ is the image of $\Pet$ under $\kappa$.  Since $\kappa(\ah_k) = \h_k$, this agrees with the definitions in Section \ref{sec:algebra}.

\begin{conjecture}\label{conj:finite}\
\begin{enumerate}
\item
The finite Fomin-Stanley algebra $\B$ satisfies Conjecture \ref{conj:LP}.
\item
The image $\kappa(j_w^0) \in \B$ of the $j$-basis element $j_w^0 \in \aB$ is a nonnegative integral linear combination of the $b_I$-basis.
\item
If $w \in \aW^0$ is such that there is a Dynkin diagram automorphism $\omega$ so that $\omega(w) \in W$, then $\kappa(j_w^0)$ belongs to the $b_I$-basis.
\end{enumerate}
\end{conjecture} 

\begin{example}
Let $W = S_4$ and $\aW = \tS_4$.  Then $\B$ is described in Example \ref{ex:S4}.  One calculates that $s^{(k)}_{221} = s_{221}+s_{32}$ so that
$$
\kappa(\s^{(k)}_{221}) = \s_{221}+\s_{32} = A_{32132}+A_{23123} \in \B.
$$
This supports Conjecture \ref{conj:finite}(2), and shows that $\kappa(j_w^0)$ does not have to be equal to 0 or to some $b_I$.
\end{example}

In the following we allow ourselves some geometric arguments and explicit calculations in other types to provide evidence for Conjecture \ref{conj:finite}.
\begin{example}
Let $W = G_2$ with long root $\alpha_1$ and short root $\alpha_2$.  Thus $\alpha_0$ is also a long root for $\tilde G_2$.  As usual we shall write $A_{i_1 i_2 \cdots i_\ell}$ for $A_{s_{i_1}s_{i_2} \cdots s_{i_\ell}}$ and similarly for the $j$-basis.  The affine Grassmannian elements of length less than or equal to 5 are 
$$
\id, s_0, s_1s_0, s_2s_1s_0, s_1s_2s_1s_0, s_2s_1s_2s_1s_0,s_0s_1s_2s_1s_0.
$$
And the $j$-basis is given by
\begin{align*}
j_\id^0 &=1 \\
j_0^0 &= A_0 + 2A_1 + A_2 \\
j_{10}^0 &= \frac{1}{2}(j_0^0)^2 \\
j_{210}^0 &= \frac{1}{2}(j_0^0)^3 \\
j_{1210}^0 &= \frac{1}{4}(j_0^0)^4 \\
j_{21210}^0&= A_{21210} + A_{21201}+2A_{21012}+A_{12102}+3A_{12101}+2A_{12121}\\
&+A_{12012}+A_{02121}+3A_{01201}+A_{01212} \\
j_{01210}^0 &= A_{01210}+A_{01201}+A_{02121}+A_{12012}+A_{12101}+A_{12102}+A_{21012}\\ &+A_{21201}+A_{21212} 
\end{align*}
which can be verified by using Theorem \ref{thm:j0basis}.  Note that $j_{21210}^0+j_{01210}^0 = \frac{1}{4}(j_0^0)^5$.  Thus $\B$ has basis
$$
\id, \ 2A_1 + A_2,  \ A_{12}+A_{21}, \ 2A_{121}+A_{212}, \ A_{1212}+A_{2121}, \ 2A_{12121}, \ A_{21212}, \ ?A_{121212}
$$
where the coefficient of $A_{121212}$ depends on the $j$-basis in degree 6.
The root poset of $G_2$ is $\alpha_1,\alpha_2 \prec \alpha_1+\alpha_2 \prec \alpha_1 + 2\alpha_2 \prec \alpha_1 + 3\alpha_2 \prec 2\alpha_1+3\alpha_2$.  Both Conjectures \ref{conj:LP} and \ref{conj:finite} hold with this choice of basis.
\end{example}

\begin{prop}\label{prop:gen}
Over the rationals, the finite Fomin-Stanley subalgebra $\B \otimes_\Z \Q$ has a set of generators in degrees equal to the exponents of $W$.
\end{prop}
\begin{proof}
Let $\Gr_G$ denote the affine Grassmannian of the simple simply-connected complex algebraic group with Weyl group $W$.
It is known \cite[Theorem 5.5]{Lam:Schub} that $\aB \simeq H_*(\Gr_G,\Z)$.  But over the rationals the homology $H_*(\Gr_G,\Q)$ of the affine Grassmannian is known to be generated by elements in degrees equal to the exponents of $W$, see \cite{Gin}.  (In cohomology these generators are obtained by {\it transgressing} generators of $H^*(K,\Q)$ which are known to correspond to the degrees of $W$.)  Since these elements generate $\aB$, their image under $\kappa$ generate $\B$.
\end{proof}

\subsection{Problems}
\begin{enumerate}
\item
Find a geometric interpretation for the finite Fomin-Stanley algebras $\B$ and the conjectural basis $b_I$.
\item
Find an equivariant analogue of the finite Fomin Stanley algebra $\B$, with the same relationship to $\B$ as $\Pet$ has to $\aB$.  Extend Conjectures \ref{conj:LP} and \ref{conj:finite} to the equivariant setting.
\end{enumerate}

\section{Geometric interpretations}
\label{sec:geom}
In this section we list some geometric interpretations of the material we have discussed.  Let $G$ be the simple simply-connected complex algebraic group associated to $W$.  (Co)homologies are with $\Z$-coefficients.

\begin{enumerate}
\item 
The affine Fomin-Stanley subalgebra $\aB$ is Hopf-isomorphic to the homology $H_*(\Gr_G)$ of the affine Grassmannian associated to $G$.  The $j$-basis $\{j_w\}$ is identified with the Schubert basis.  \cite{Pet} \cite[Theorem 5.5]{Lam:Schub}
\item
The affine Schur functions $\tF_\la$ represent Schubert classes in $H^*(\Gr_{SL(n)})$.  \cite[Theorem 7.1]{Lam:Schub}
\item
The affine Stanley symmetric functions are the pullbacks of the cohomology Schubert classes from the affine flag variety to the affine Grassmannian. \cite[Remark 8.6]{Lam:Schub}
\item
The positivity of the affine Stanley to affine Schur coefficients is established via the connection between the homology of the $\Gr_G$ and the quantum cohomology of $G/B$.  \cite{Pet} \cite{LS:QH} \cite{LL}
\item
The (positive) expansion of the affine Schur symmetric functions $\{\tF_\la\}$ for $\aW = \tS_n$ in terms of $\{\tF_\mu\}$ for $\aW = \tS_{n-1}$ has an interpretation in terms of the map $H^*(\Omega SU(n+1)) \to H^*(\Omega SU(n))$ induced by the inclusion $\Omega SU(n) \hookrightarrow \Omega SU(n+1)$ of based loop spaces.  \cite{Lam:ASP}
\item
Certain affine Stanley symmetric functions represent the classes of {\it positroid varieties} in the cohomology of the Grassmannian. \cite{KLS}
\item
The expansion coefficients of Stanley symmetric functions in terms of Schur functions are certain quiver coefficients. \cite{Buc2}
\end{enumerate}

\end{document}